\documentclass[a4paper]{article}

\usepackage[english]{babel}

\usepackage[utf8]{inputenc}
\usepackage[T1]{fontenc}
\usepackage{lmodern}

\usepackage{tikz}  
\usetikzlibrary{matrix} 
\pgfdeclarelayer{background} 
\pgfsetlayers{background,main} 

\usepackage[backend=bibtex,style=alphabetic,firstinits=true,isbn=false,doi=false,url=false]{biblatex}

\usepackage{amsmath, amsthm, amsfonts, amssymb, color, upgreek,bm}
\usepackage{bbm}
\usepackage{tikz}
\usepackage{graphicx}

\usepackage{wasysym }
\usepackage{mathrsfs}

\usepackage{todonotes}
\usepackage{mathtools,bbm}
\usepackage[normalem]{ulem}
\usepackage{cancel}
\usepackage{verbatim}
\usepackage{tikz}
\usetikzlibrary{matrix}
\allowdisplaybreaks[4]
\usepackage[shortlabels]{enumitem}
\setenumerate{label={\normalfont(\alph*)},topsep=4pt,itemsep=0pt} 

\usepackage{hyperref}
\usepackage{cleveref}
\usepackage{a4}
\usepackage{multirow}
\usepackage[footnotesize,bf,centerlast]{caption}
\usepackage[small,euler-digits,icomma,OT1,T1]{eulervm}

\usepackage{xcolor,colortbl}
\definecolor{Gray}{gray}{0.80}
\definecolor{LightGray}{gray}{0.90}

\newcommand{\cD}{\mathcal{D}}

\newcommand{\cH}{\mathcal{H}}

\newcommand{\cL}{\mathcal{L}}

\newcommand{\cO}{\mathcal{O}}

\newcommand{\bR}{\mathbb{R}}

\newcommand{\bONE}{\mathbbm{1}}

\newcommand{\dd}{ \mathrm{d}}

\renewcommand{\epsilon}{\varepsilon}
\newcommand{\vn}[1]{\left|  \left| #1\right|  \right|}

\numberwithin{equation}{section}

\newtheorem{theorem}{Theorem}[section]
\newtheorem{lemma}[theorem]{Lemma}
\newtheorem{proposition}[theorem]{Proposition}
\newtheorem{corollary}[theorem]{Corollary}

\theoremstyle{definition}
\newtheorem{definition}[theorem]{Definition}

\newtheorem{remark}[theorem]{Remark}

\newtheorem{assumption}[theorem]{Assumption}

\title{Large deviations for slow-fast processes on connected complete Riemannian manifolds}

\author{
	{Yanyan Hu$^{a)}$, Richard C. Kraaij$^{a)}$, Fubao Xi$^{b)}$}\\
\footnotesize$^{a)}${Delft Institute of Applied Mathematics, Delft University of Technology,}\\
\footnotesize{Mekelweg 4, 2628 CD,  Delft, The Netherlands.} \\
\footnotesize$^{b)}${School of Mathematics and Statistics, Beijing Institute of Technology, Beijing 100081, P.R.\ China}\\
\footnotesize$^{*}${Corresponding author's email: y.hu-2@tudelft.nl}
}
\date{}
\begin{document}

\maketitle
\begin{abstract}
We consider a class of slow-fast processes on a connected complete Riemannian manifold $M$.
The limiting dynamics as the scale separation goes to $\infty$ is governed by the averaging principle. Around this limit, we prove large deviation principles with an action-integral rate function for the slow process by nonlinear semigroup methods together with the Hamilton-Jacobi-Bellman equation techniques. The innovation is solving a comparison principle for viscosity solutions on $M$ and the existence of a viscosity solution via a control problem for a non-smooth Hamiltonian.
\\[0.2cm]

\end{abstract}




\section{Introduction}\label{se1}
In this paper, let $M$ be a $d$-dimensional connected complete Riemannian manifold. We consider a stochastic differential equation on 
$M$ with initial value $(x_0,k_0)$:
\begin{equation}\label{eqn_slowCIR}   \mathrm{d}X^\varepsilon_n(t)=\frac{1}{\sqrt{n}}u^\varepsilon_n(t) \circ\mathrm{d}W(t)+b(X^\varepsilon_n(t),\Lambda^\varepsilon_n(t))\mathrm{d}t,
\end{equation}
where $\Lambda_n^\varepsilon(t)$ is a switching process with transition rate on a set $S=\{1,2,\ldots,N\}$, $N<\infty$,
\begin{equation}\label{eqn_fastswitching}
	\mathbb{P} (\Lambda^{\varepsilon}_n(t+\triangle)=j~|~\Lambda^{\varepsilon}_n(t)=i, X^{\varepsilon}_n(t)= x)=
	\begin{cases}
		\frac{1}{\varepsilon}q_{ij}(x)\triangle+ \circ(\triangle) ,&\mbox{if $j\neq i$,}\\
		1+\frac{1}{\varepsilon}q_{ij}(x)\triangle +\circ(\triangle),&\mbox{if $j=i$,}
	\end{cases}
\end{equation}
for small $\Delta >0$, $i,j \in S$, $x\in M$, and $\varepsilon>0$ is a small parameter. $u^\varepsilon_n(\cdot)$ is a unique element such that $X^\varepsilon_n(t)=\mathbf{p}u^\varepsilon_n(t)$, where $\mathbf{p}:O(M)\to M$ is a projection map. Precise details and conditions on this system will be specified later. Obviously, \eqref{eqn_slowCIR} and \eqref{eqn_fastswitching} together is a slow-fast system. 
\par
It is not too difficult to see that under some conditions, the effective behavior of the slow process 
\eqref{eqn_slowCIR} can be accurately described by the averaged system as $\varepsilon\to 0$ and $n\to \infty$, utilizing the averaging principle.
To be precise, given a fixed slow process and subject to appropriate conditions, the fast process has an invariant probability measure. This observation implies that, as $\varepsilon\to 0$ and $n\to \infty$, the slow process converges to an averaged process defined as follows
\begin{equation*}
    \mathrm{d}\bar{X}(t)=\bar{b}(\bar{X}(t))\mathrm{d}t,
\end{equation*}
where $\bar{b}(x)=\sum_{i\in S}b(x,i)\pi^x_i(t)$ and $\pi^x(t)=(\pi^x_i(t))_{i\in S}$ is the unique invariant probability measure of the fast process with the slow variable being ``frozen'' at a deterministic point $x\in M$.
The application of this averaging principle provides an effective method to reduce computational complexity. It can be viewed as a variant of the law of large numbers.
In contrast to the averaging principle, the large deviation principle (LDP) excels in providing a more precise capture of dynamical behavior. In other words, it specifically addresses the characterization of the exponential decay rate associated with probabilities of rare events.
\par
The main purpose of this paper is to prove a LDP around such averaged process on $M$. The theory of LDP is one of the classical topics in probability theory, for example, \cite{DZ1998,DH2008,FK2006}, which has widespread applications in different areas such as information theory, thermodynamics, statistics, and engineering.
\par
Let us mention some works related to our purposes. Huang, Mandjes and Spreij
\cite{HM2016}, they studied large deviations for Markov-modulated diffusion processes with rapid switching. 
In \cite{PS2021}, Peletier and Schlottke proved pathwise LDP of switching Markov processes by exploiting the connection between Hamilton-Jacobi (HJ) equations and Hamilton-Jacobi-Bellman (HJB) equations.
In \cite{KS2020}, Kraaij and Schlottke studied the LDP for the slow-fast system under regular conditions, where the fast process is a switching process. For the proof, they used the Bootstrapping procedure, which is a technology for comparison principle of the HJB equation. Later, Della Corte and Kraaij \cite{CK2024} continued to explore LDP in the context of molecular motors modeled by a diffusion process driven by the gradient of a weakly periodic potential that depends on an internal degree of freedom. The switch of the internal state, which can freely be interpreted as a molecular switch, is modeled as a Markov jump process that depends on the location of the motor. Subsequently, Hu, Kraaij, and Xi \cite{HKX2023} considered the  Cox-Ingersoll-Ross processes  with state-dependent fast switching in the case of degenerate.
\par
Although there are extensive results on LDPs for slow-fast systems on Euclidean space, there is not much work on the topic of Reimannian manifold. 
R\"{o}ckner and Zhang \cite{RZ2004} studied sample path large deviations for diffusion processes on configuration spaces over a Riemannian manifold. Kraaij, Redig and Versendaal \cite{KRV2019} generalized classical large deviation theorems to the setting for complete, smooth Riemannian manifolds. However, they focused on the simple setting of random walks rather than slow-fast systems. 
Versendaal \cite{V2020} studied large deviations for Brownian motion in evolving Riemannian manifolds.
\par
Motivated by the aforementioned papers about LDP for slow-fast processes on Euclidean space and simple LDP on Riemannian manifold, it is a natural question to ask how to generalize on the Riemannian manifold for the slow-fast processes. In this paper, we reply to this question. That is we
prove LDPs with an action-integral rate function for the slow process by nonlinear semigroup methods together with the HJB equation techniques. Note that our drift coefficient of slow process only satisfies locally one-sided Lipschitz continuity, which is weaker than the bounded condition. Moreover, the rate functions are related to the Hamiltonian $\mathcal{H}:T^*M\to \mathbb{R}$ obtained by taking the Lagrangian transforms of $\mathcal{L}:TM\to \mathbb{R}$. One formally defines that
\begin{equation}\label{eqn_bfHH}
\begin{split}
	\mathbf{H}f(x)=\mathcal{H}(x,\mathrm{d}f(x))=\sup_{\pi\in\mathcal{P}(S)}\big\{\int_M B_{x,\mathrm{d}f(x)}(z)\pi(\text{d}z)-\mathcal{I}(x,\pi)\big\},
\end{split}
\end{equation}
where
\begin{equation*}
	B_{x,\mathrm{d}f(x)}(z)=b(x,z) \mathrm{d} f(x)+\frac{1}{2}
	\left|\mathrm{d}  f(x)\right|^2
\end{equation*}
coming from the slow process $X_n(t)$ and Donsker-Varadhan function
\begin{equation*}
	\mathcal{I}(x,\pi)=-\inf_{g>0}\int_M\frac{R_xg(z)}{g(z)}\pi(\text{d}z),
\end{equation*}
where $R_x$ is the generator corresponding to the fast process $\Lambda_n(t)$ defined by 
\begin{equation*}
	R_xg(z)=\sum_{j\in S }q_{zj}(x)\left(g(j)-g(z)\right).
\end{equation*}	
Although following the proof ideas from Feng and Kurtz’s book \cite{FK2006}, considering the comparison principle and the existence of HJB equations, we need to put forward some new tricks to show those owing to the special properties of the Riemannian manifold.
\par
We first discover special properties on $M$, which have caused difficulties but also is the key innovation in our proof:
\begin{enumerate}[(i)]
    \item The first one, to ensure the exponential tightness, we find a good containment function:
    \begin{equation*}
\Upsilon(x)=\frac{1}{2}\log(1+f^2(x)), 
\end{equation*}
where the smooth function $f(x)$ approximates $d(x_0,x)$ for some $x_0\in M$ and satisfying formally $\sup _z \mathcal{H}(z,\mathrm{d}\Upsilon(z))<C<\infty$ which plays the role of a relaxed Lyapunov. 
 \item The second one, the distance function $d(x,y)$, $x$, $y\in M$ is no longer smooth, although $x$ is close to $y$. This happens because the shortest path (geodesic) between two points may not be unique, for example, a sphere $\mathbb{S}^d$, $d$ is an integer. Cut-locus is essentially the boundary of the region that can be reached by geodesics originating from a specific point. 
    \item The third one, we need to prove the global existence of solutions for a HJB equation on $M$ to obtain an action-integral rate function. To establish existence we need to solve an appropriate control problem. A key obstacle is the construction from local solutions to global solutions.
\end{enumerate}
\par
\textbf{Organization:}
The organization of our paper is as follows:
in \Cref{sec_statement_main_result}, we introduce fundamental concepts related to the large deviation principle.
In \Cref{sec_establishing_diffusion}, we construct a diffusion process with fast switching on the Riemannian manifold, and state our main results. Subsequently, in \Cref{sec:The_strategy}, we articulate the strategy employed in proving the large deviation. In \Cref{sec:the_proof_of_main_theorem}, We provide a comprehensive proof of the main theorem concerning the large deviation.

\section{Preliminaries}\label{sec_statement_main_result}
The following convention will be used throughout the paper: $C$ and $c$ with or without indices will denote diﬀerent positive constants whose values may change from one place to another.
\par
We begin with the necessary definitions for introducing the large deviation principle on Riemannian manifold $M$. The basic knowledge about Riemannian manifold is useful in this paper, and we have placed it in \Cref{se_Riemannian_manifold} to highlight our main results.
\begin{definition}
    Consider a sequence of $X_1,X_1,\ldots$ on complete Riemannian manifold $M$. Furthermore let $I: F\to [0,\infty]$.
    \begin{enumerate}
    \item We say that $I$ is a good rate function if for every $c\geq 0$, the set $\{x~|~I(x)\leq c\}$ is compact.
        \item We say that the sequence $\{X_n\}_{n\geq 1}$ is \textit{exponentially tight} if for all $\alpha>0$ there exists a compact set $K_\alpha \subseteq F $ such that 
        \begin{equation*}
            \limsup_{n\to \infty}\frac{1}{n}\log X_n(K^c_\alpha)<-\alpha.
        \end{equation*}
        \item We say that the sequence $\{X_n\}_{n\geq 1}$ satisfies the \textit{large deviation principle} with rate $n$ and good rate function $I$, denoted by
        \begin{equation}
            \mathbb{P}[X_n \backsim a]\approx \mathrm{e}^{-nI(a)},
        \end{equation}
        \begin{enumerate}[(i)]
        \item if we have for every closed set $A\subseteq F$ the upper bound,
            \begin{equation*}
                \limsup_{n\to \infty} \frac{1}{n}\log\mathbb{P}[X_n\in A]\leq -\inf_{x \in F}I(x).
            \end{equation*}
                \item and for every open set $U\subseteq F$ the lower bound,
            \begin{equation*}
                \liminf_{n\to \infty}\frac{1}{n}\log \mathbb{P}[X_n\in U]\geq -\inf_{x\in U}I(x).
            \end{equation*}
            \end{enumerate}
            \end{enumerate}
\end{definition}
\begin{definition}[Absolutely continuous curves]
	We denote by $\mathcal{AC}(M)$ the space of absolutely continuous curves in $M$. 
	A curve $\gamma:[0,T]\to M$ is absolutely continuous if there exists a function $g\in L^1[0,T]$ such that for $t\in[0,T]$ we have $\gamma(t)=\gamma(0)+\int^t_0g(s)\text{d}s$. We write $g=\dot{\gamma}$.
	
	A curve $\gamma:[0,\infty)\to M$ is absolutely continuous, i.e. $\gamma\in \mathcal{AC}(M)$, if the restriction to $[0,T]$ is absolutely continuous for every $T>0$.
\end{definition}

\section{Constructing a diffusion process with fast switching on Riemannian manifold}\label{sec_establishing_diffusion}

In the above section, we only gave the basic knowledge about the large deviation principle. We next state the definition of the orthonormal frame bundle and horizontal lift to construct a diffusion process with switching on $M$ as a model to study large deviation principle, i.e., the slow-fast systems \eqref{eqn_slowCIR} and \eqref{eqn_fastswitching}.
\par
Let $O_x(M)$ be the space of all orthonormal bases of $T_xM$. Denote
$O(M)~:~=\cup_{x\in M} O_x(M)$, which is called the \textit{orthonormal frame bundle} over $M$. Obviously, $O_x(M)$ is isometric to $O(d)$, the group of orthogonal $(d \times d)$-matrices.  
\par
To see that $O(M)$ has a natural Riemannian structure, let $\mathbf{p}: O(M)\to M$ with $\mathbf{p}u := x$ if $u\in O_x(M)$, which is called the \textit{canonical projection} from $O(M)$ onto $M$.
Now, given $e\in \mathbb{R}^d$, our goal is to define the corresponding horizontal vector field on $O(M)$. On the one hand, for any $u \in O(M)$ we have $ue \in T_{\mathbf{pu}}M$. Let $u_s$ be the parallel transportation of u along the geodesic $\exp_{\mathbf{p}u}(sue)$, $s \geq 0$. We obtain a vector 
\begin{equation*}H_e (u):=\frac{\mathrm{d}}{\mathrm{d}s} u_s|_{s=0} \in T_u O(M).
\end{equation*} Thus, we have defined a vector filed $H_e$ on $O(M)$ which is indeed $C^\infty$-smooth. In particular, let $\{e_i\}_{i=1}^d$ be an orthonormal basis on $\mathbb{R}^d$, deﬁne
\begin{equation*}
\Delta_{O(M)}:=\sum^{d}_{i=1}H^2_{e_i}.
\end{equation*}
It is easy to see that this operator is independent of the choice of the basis $\{e_i\}$. We call $\Delta_{O(M)}$ \textit{the horizontal Laplace operator}. On the other hand, for any vector field $Z$ on $M$, we define its \textit{horizontal lift} by $\mathbf{H}_{Z}(u):=H_{u^{-1}Z}(u)$, $u\in O(M)$, where $u^{-1}Z$ is the unique vector $e\in \mathbb{R}^d$ such that $Z_{\mathbf{p}u}=ue$.
\par
Let $\Delta_M$ be a Laplace-Beltrami operator,
\begin{equation}\label{eqn_Laplace_Beltrami_operator}
	\Delta_M f=\frac{1}{\sqrt{G}}\frac{\partial}{\partial x^i}\left(\sqrt{G} g^{ij}\frac{\partial f}{\partial x^j} \right),~~~f\in C^1(M).
\end{equation}
We have the conclusion below, the horizontal Laplacian $\Delta_{O(M)}$ is the lift of the Laplace-Beltrami operator $\Delta_M$ to the orthonormal frame bundle $O(M)$. 
\begin{lemma}[Proposition 3.1.2 of \cite{H2002MR1882015}]
	Let $f\in C^{\infty}(M)$, and $\Tilde{f}=f\circ \mathbf{p}$ its lift to $O(M)$. Then for any $u\in O(M)$,
	\begin{equation*}
		\Delta_Mf(x)=\Delta_{O(M)}\Tilde{f}(u),
	\end{equation*}
	where $x=\mathbf{p}u$.
\end{lemma}
\par
Having the preparations of orthonormal frame bundle and horizontal lift, we can establish a diffusion process \eqref{eqn_slowCIR} with switching \eqref{eqn_fastswitching} on $M$ in detail. To this end, we divide it into two steps:

\underline{Step 1: A SDE with a fixed switching state.} 

Let $b:\mathbb{R}^d\to TM$ be a $C^1$-smooth vector field on $M$. According to the idea of \cite[Section 2.1]{W2014MR3154951}, we study a diffusion process generated by $A^{M}_n:=\frac{1}{2n}\Delta_M+b$, where $\Delta_M$ is a Laplace-Beltrami operator in \eqref{eqn_Laplace_Beltrami_operator}.

	To this end, we first construct the corresponding \textit{Horizontal diffusion process} generator by 
	$A^{O(M)}_n:=\frac{1}{n}\Delta_{O(M)}+\mathrm{H}_{b}$ on $O(M)$ by solving the
	Stratonovich stochastic differential equation 
	\begin{equation*}
		\mathrm{d}u_n(t)=\frac{1}{\sqrt{n}}\sum_{j=1}^d H_{e_j}(u_n(t))\circ \mathrm{d}W^i(t)+{H}_{b}(u_n(t))\mathrm{d}t,~~u_n(0)=u\in O(M),
	\end{equation*}
	where $W(t):=(W^1(t),\ldots,W^d(t))$ is the $d$-dimensional Brownian motion on a complete filtered probability space $(\Omega,~\mathcal{F},~\{\mathcal{F}_t\}_{t\geq 0},~\mathbb{P})$. Since ${H}_b$ is $C^1$, it is well known that (see e.g. \cite[Chapter IV, Section 6]{E1982MR675100}) the equation has a unique solution up to the life time $\zeta:= \lim_{j\to \infty} \zeta_j$, where
	\begin{equation*}
		\zeta_j:=\inf\{t\geq 0:\mathrm{d}(\mathbf{p}u,~\mathbf{p}u_n(t))\geq j\},~~j\geq 1.
	\end{equation*}
	
	Let $X_n(t)=\mathbf{p} u_n(t)$. Then $X_n(t)$ solves the equation
	\begin{equation}\label{eqn_SDEmanifold}
		\mathrm{d}X_n(t)=\frac{1}{\sqrt{n}}u_n(t) \circ\mathrm{d}W(t)+b(X_n(t))\mathrm{d}t,~~~X_n(0)=x_0: =\mathbf{p}u
	\end{equation}
	up to the lifetime $\zeta$.
	By the It\^{o} formula, for any $f(\cdot)\in C^2_0(M)$,
	\begin{equation*}
		f(X_n(t))-f(x_0)-\int^t_0A_n^{M}f(X_n(s))\mathrm{d}s=\frac{1}{\sqrt{n}}\int^t_0\langle (u_n(s))^{-1
		}\mathrm{grad}f(X_n(s)),\mathrm{d}W(s)\rangle
	\end{equation*}
	is a martingale up to the life time $\zeta$; that is $X^\varepsilon_n(t)$ is the diffusion process generated by $A^M_n$, and we call it the $A^M_n$-diffusion process. When $b=0$, then $X_n(t)$ is generated by $\frac{1}{2n}\Delta_M$ and is called the Brownian motion on $M$.
	
	\underline{Step 2: the SDE with switching for any states.} 
	Here, we are going to introduce SDE with switching in \eqref{eqn_SDEmanifold}. To do it, for $S=\{1,2\ldots,N\}$, $N<\infty$, we let the drift coefficient of the slow process depend on $i\in S$, where $i$ stands for the state of the switching process. For the SDE with the fixed $i$, one can refer to step 1. But for any $i \in S$, we need an infinite lifetime referring to the process.
	\par
	With the short and useful analysis in our mind, we cite \Cref{lem_Mass_conservation} below from the book \cite[Theorem 3.2.6]{BGL2014MR3155209} with $\rho=\rho(n)$ and $b=Z$ to guarantee infinite lifetime.
	\begin{lemma}[Mass conservation \cite{BGL2014MR3155209}]\label{lem_Mass_conservation}
		Let $\mathcal{L}$ be an elliptic diffusion operator with semigroup $\mathbf{P}=(P_t)_{t\geq 0}$, symmetric with respect to a positive measure $\mu$, on a smooth complete connected manifold $M$. If the $CD(\rho,\infty)$ curvature condition 
		\begin{equation*}
			\mathcal{R}(L)=\mathcal{R}_g-\nabla Z\geq \rho g
		\end{equation*}
		holds for some $\rho\in \mathbb{R}$, then for all $t\geq 0$, $P_t(\bONE)=\bONE$.
		$\mathcal{R}_g$ is the Ricci tensor of the (co)-metric $g$. 
	\end{lemma}
	
In this setting, \eqref{eqn_SDEmanifold} with initial value $(X^\varepsilon_n(0),~\Lambda^\varepsilon_n(0))=(x_0,k_0)$ becomes
\begin{equation}\label{eqn_goalSDE}
	\mathrm{d}X^\varepsilon_n(t)=\frac{1}{\sqrt{n}}u^\varepsilon_n(t) \circ\mathrm{d}W(t)+b(X^\varepsilon_n(t),\Lambda^\varepsilon_n(t))\mathrm{d}t,
\end{equation}
where $\Lambda_n^\varepsilon(t)$ is a switching process with transition rate 
\begin{equation}\label{eqn_goalswitch}
	\mathbb{P} (\Lambda^{\varepsilon}_n(t+\triangle)=j~|~\Lambda^{\varepsilon}_n(t)=i, X^{\varepsilon}_n(t)= x)=
	\begin{cases}
		\frac{1}{\varepsilon}q_{ij}(x)\triangle+ \circ(\triangle) ,&\mbox{if $j\neq i$,}\\
		1+\frac{1}{\varepsilon}q_{ij}(x)\triangle +\circ(\triangle),&\mbox{if $j=i$,}
	\end{cases}
\end{equation}
for small $\Delta >0$, $i,j \in S$, $x\in M$, and $\varepsilon>0$ is a parameter. The system 
\eqref{eqn_goalSDE} and \eqref{eqn_goalswitch} is a diffusion system with state-dependent fast switching.
\subsection{The main results}
In this paper, we consider the slow-fast systems \eqref{eqn_goalSDE} and \eqref{eqn_goalswitch}.
We first collect all the assumptions that are needed before giving the main results.

\begin{assumption}\label{assu_varn_manifold}
	Let $\varepsilon=\frac{1}{n}$, this shows that small disturbance and fast switching have the same rate.
\end{assumption}

This assumption means that the slow-fast system
      \eqref{eqn_goalSDE} and \eqref{eqn_goalswitch} becomes 
\begin{equation}\label{eqn_1goalSDE}
    \mathrm{d}X_n(t)=\frac{1}{\sqrt{n}}u_n(t) \circ\mathrm{d}W(t)+b(X_n(t),\Lambda_n(t))\mathrm{d}t,
\end{equation}
and
\begin{equation}\label{eqn_1goalswitch}
	\mathbb{P} (\Lambda_n(t+\triangle)=j~|~\Lambda_n(t)=i,~ X_n(t)= x)=
	\begin{cases}
		nq_{ij}(x)\triangle+ \circ(\triangle) ,&\mbox{if $j\neq i$,}\\
		1+nq_{ij}(x)\triangle +\circ(\triangle),&\mbox{if $j=i$.}
	\end{cases}
\end{equation}
In the following, we will focus on \eqref{eqn_1goalSDE} and \eqref{eqn_1goalswitch}.
\begin{assumption}\label{ass_b_linear_growth} Fix $x_0\in M$ and define $r(x)=d(x,x_0)$. We say that $b$ is linear growth if there exists a constant $C>0$ such that, for all $x\in M$,
\begin{equation*}
    |b(x,i)|\leq C(1+r(x)),~~~\forall~i\in S.
\end{equation*}
\end{assumption}

\begin{assumption}\label{ass_b_one_side}
    We say that $b$ is locally one-sided Lipschitz function if for any compact sets $K\subseteq M$, there exists a constant $C_K> 0$ such that, for all $x$, $y\in K$, it holds that \begin{equation*}
    \mathrm{d}_x\left(\frac{1}{2}d^2(\cdot, y)\right)(x)b(x,i)-\mathrm{d}_y\left(-\frac{1}{2}d^2(x,\cdot)\right)(y)b(y,i)\leq C_Kd^2(x,y),~~~\forall i\in S,
    \end{equation*}
    where $d(x,y)<i(K)$ and $i(K)$ is the injectivity radius of $K$ defined in \Cref{se_Riemannian_manifold}.
\end{assumption}

\begin{assumption}\label{asm_conservative_manifold}
	For any $x\in M$, $(q_{ij}(x))_{i,j\in S}$ is a conservative, irreducible transition rate matrix, and $\sup_{i\in S}q_i(x)<\infty$, where $q_i(x)=-q_{ii}(x)=\sum_{j\in S,j\neq i}q_{ij}(x)$.
\end{assumption}

\begin{assumption}\label{asm_conti_manifold}
	For any compact sets $K\subseteq M$, there exists a constant $C_K>0$ such that 
	\begin{equation*}
		|q_{ij}(x)-q_{ij}(y)|\leq C_Kd(x,y),~~x,~y\in K,~i,~j \in S.
	\end{equation*}
\end{assumption}

\begin{assumption}\label{ass_bfH_bounded}
    Let $M$ be a Riemannian manifold, there exists a operator $\mathbf{H}\subseteq C_b^1(M)\times C_b(M)$ with domain
    \begin{equation*}
        \mathcal{D}(\mathbf{H})=\{f\in C^1_b(M)~| ~ c_f = \sup_{z} \left\{ -\mathbf{H}f(z) \right\} < \infty \}.
    \end{equation*}
\end{assumption}
Then, we give some remarks on these assumptions.
\begin{itemize}
    
   \item \Cref{ass_b_linear_growth} controls the rate at which the process may deviate to prove exponential tightness.
  \item  \Cref{ass_b_one_side} is set for proving the comparison principle.
\item
 Assumptions \ref{asm_conservative_manifold} and \ref{asm_conti_manifold} of a fast switching process for any given $x$ ensures the existence of an invariant probability measure that satisfies the averaging principle.
\item \Cref{ass_bfH_bounded} is a regularity condition to obtain optimally controlled curves that are needed in the proof of variational representation of rate function.
\end{itemize}
Although we have constructed the slow-fast process at the beginning of this section, we emphasize the well-posedness of this process with a theorem.
\begin{theorem}
	The system, \eqref{eqn_1goalSDE} and \eqref{eqn_1goalswitch}, has a  unique non-explosive strong solution $(X_n(t),\Lambda_n(t))$ with initial value $(X_n(0),\Lambda_n(0))=(x_0,k_0)$. 
\end{theorem}
\begin{proof}
	By constructing, we first have  the well-posedness of \eqref{eqn_SDEmanifold} motivating by the proof of the \cite[Proposition 2.4]{HKX2023}, then use the method of sequentially splicing solutions on the trajectory with switching times implying the well-posedness of the solution for \eqref{eqn_1goalSDE} and \eqref{eqn_1goalswitch}. Here, we omit the detailed proof. 
\end{proof}
In the following, we give the main result. 

\begin{theorem}[Large deviation principle]\label{thm_LDPmanifold} 
Let $(X_n(t),\Lambda_n(t))$ be the Markov processes on $M\times S$. Consider the setting of Assumptions \ref{assu_varn_manifold}, \ref{ass_b_linear_growth}, \ref{ass_b_one_side}, \ref{asm_conservative_manifold}, \ref{asm_conti_manifold} and \ref{ass_bfH_bounded}.
	Suppose that 
	the large deviation principle holds for $X_n(0)$ on $M$ with speed $n$ and a good rate function $I_0$.
\par
Then, the large
deviation principle is satisfied with speed $n$ for the processes $X_n(t)$ with a good rate function $I$ having action-integral representation,
\begin{equation*}
	I(\gamma)=
	\begin{cases}
		I_0(\gamma(0))+\int^\infty_0\mathcal{L}\left(\gamma(s),\dot{\gamma}(s)\right)\mathrm{d}s,&\mbox{if~$\gamma\in\mathcal{AC}(M)$,}\\
		\infty,&\mbox{otherwise.}
	\end{cases}
\end{equation*}
where $\mathcal{L}: TM \to  [0, \infty]$ is the Lagrangian transform of $\mathcal{H}$ given by $\mathcal{L}(x, v) = \sup_{p\in T^*_xM} \{\langle v, p\rangle - \mathcal{H}(x, p)\} $,
and \begin{equation*}
\begin{split}
\mathcal{H}(x,\mathrm{d}f(x))=\sup_{\pi\in\mathcal{P}(S)}\big\{\int_M B_{x,\mathrm{d}f(x)}(z)\pi(\text{d}z)-\mathcal{I}(x,\pi)\big\}
\end{split}
\end{equation*}
is given in \eqref{eqn_bfHH}.
\end{theorem}

\section{The strategy of proof \Cref{thm_LDPmanifold}}\label{sec:The_strategy}
In this section, we begin with the necessary preparation for the coming strategy.
The semigroups of log-Laplace transforms of the conditional probabilities
\begin{equation}\label{eqn_Vsemigroup}
V_n(t)f(x_0,k_0)=\frac{1}{n}\text{\rm{log}}\mathbb{E}[\mathrm{e}^{nf(X_n (t),\Lambda_n(t))}~|~(X_n(0),\Lambda_n(0))=(x_0,k_0)]
\end{equation}
have generators 
\begin{equation}\label{eqn_HVN}
    H_nf=\frac{\mathrm{d}}{\mathrm{d}t}V_n(t)f \Big|_{t=0} =\frac{1}{n}\mathrm{e}^{-nf}A^M_n\mathrm{e}^{nf}.
\end{equation}
Recall that $\mathcal{L}: TM \to  [0, \infty]$, defined as 
\begin{equation*}
	\mathcal{L}(x, v) = \sup_{p\in T^*_xM}\{\langle p, v\rangle-\mathcal{H}(x,~p)\}.   
\end{equation*}
This Lagrangian keeps track of the cost along a trajectory that will play a central role in the form of the rate function of the large deviation principle. 
Then in terms of $\mathcal{L}$, we define a \textit{variational semigroup} $\mathbf{V}(t)$, $t \geq 0$,
\begin{equation*}
	\mathbf{V}(t)f(x):=\sup_{\gamma \in\mathcal{AC}\atop \gamma(0)=x}\left\{f(\gamma(t))-\int^t_0\mathcal{L}(\gamma(s),\dot{\gamma}(s))\mathrm{d}s\right\}
\end{equation*}
and \textit{resolvent} $\mathbf{R}(\lambda)$, $\lambda> 0$,  
\begin{equation}\label{eqn_bfR}
\begin{split}
	\mathbf{R}(\lambda)h(x):=\sup_{\gamma\in \mathcal{AC}\atop \gamma(0)=x}\left\{\int^\infty_0 \lambda^{-1}e^{-\lambda^{-1}t}\left(h(\gamma(t))-\int^t_0\mathcal{L}(\gamma(r),\dot{\gamma} (r))\mathrm{d}r\right)\mathrm{d}s\right\}.
 \end{split}
\end{equation}

\begin{definition}[Operator convergence]
For any $(f,g)\in H$, there exists a sequence $(f_n,g_n)\in H_n$ such that $\|f_n-f\|\to 0$ and $\|g_n-g\|\to 0$ as $n\to \infty$.
\end{definition}

As typical in proofs of path space of LDP for Markov processes, we argue via the projection limit theorem. We repeat this classical idea for complement:
\begin{enumerate}[(1)]
    \item 
    If the processes are exponentially tight, it suffices to establish the large deviation principle for finite dimensional distributions. 
    \item The large deviation principle for finite dimensional distributions can be established by proving that the semigroup of log Laplace-transforms of the conditional probabilities converges to a limiting semigroup.
    \item One can often rewrite the limiting semigroup as a variational semigroup, which allows to rewrite the rate-function on the Skorohod space in Lagrangian form.
\end{enumerate}
In detail, the strategy to prove path space large deviation principle for a sequence of the Markov processes $(X_n(t),\Lambda_n(t))$ on $M$, formally works as follows:
\begin{enumerate}[(i)]
    \item \textit{Identification of a single value Hamiltonian}. The semigroups of log-Laplace transforms of the conditional probabilities $V_n(t)$ in \eqref{eqn_Vsemigroup} have generators $H_n$. Then one verifies that the nonlinear generator sequence $H_n$ converges to the limiting operator $H$ as $n\to \infty$.
    \item \textit{Exponential tightness on Riemannian manifold}. Provided one can verify the exponential compact containment condition, the convergence of the sequence $H_n$ gives exponential tightness.
    \item \textit{Comparison principle on Riemannian manifold}. The theory of viscosity solutions gives applicable conditions for proving that the limiting Hamiltonian generates a semigroup. If for all $\lambda> 0$ and $h\in C_b(M)$, the Hamilton-Jacobi-Bellman equation $f-\lambda Hf = h$ admits a unique solution, one can extend the generator $H$ so that the extension satisﬁes the conditions of Crandall-Liggett theorem \cite{CL1971} and thus generates a semigroup $V(t)$. Additionally, it follows that the semigroups $V_n(t)$ converge to $V(t)$, giving the large deviation principle. Uniqueness of the solution of the Hamilton-Jacobi equation can be established via the comparison principle for sub- and super-solutions.
    \item \textit{Variational representation of the Hamiltonian on Riemannian manifold}. By Lagrangian transforming the limiting Hamiltonian $H$, one can define a “Lagrangian” which can be used to define a variational semigroup and a variational resolvent. It can be shown that the variational resolvent provides a solution of the Hamilton-Jacobi-Bellman equation and therefore, by uniqueness of the solution, identiﬁes the resolvent of $H$. As a consequence, an approximation procedure yields that the variational semigroup and the limiting semigroup $V(t)$ agree. A standard argument is then sufficient to give a Lagrangian form of the path-space rate function. 
\end{enumerate}

In addition, Feng and Kurtz summarized this in their book \cite{FK2006}. We modified it to fit our content.
	\begin{proposition}[Adaptation of Theorem 5.15, Theorem 8.27 and Corollary 8.28 in \cite{FK2006} to our context]\label{proposition_bridge}
		Let $(X_n(t),\Lambda_n(t))$ be Markov processes on $M\times S$.
		Suppose that
		\begin{enumerate}
			\item  \label{item_a} $X_n(0)$  satisfies large deviation principle;
			\item \label{item_b} there exists an operator $H$ such that we have $\|H_n-H\|\to 0$ as $n\to\infty$;
			\item \label{item_c} we have exponential tightness of the process $(X_n(t),\Lambda_n(t))$;
			\item  \label{item_d} for all $\lambda>0$ and $h\in C_b(M)$, the comparison principle holds for $f-\lambda Hf=h$ with same $H$ as in \cref{item_b}, and there exists a unique solution $R(\lambda)h$;
  \item \label{item_e} $R(\lambda)h=\mathbf{R}(\lambda)h$.
		\end{enumerate}
	 
		Then the following hold:
		\begin{enumerate}[(i)]
			\item \label{item_Limit of nonlinear semigroup} (Limit of nonlinear semigroup) For each $x\in M$, $f\in \overline{\mathcal{D}(H)}$, $f_n\in \mathcal{B}(M\times S)$ and $t\geq 0$, if $\|f_n-f\|\to 0$ as $n\to \infty$, there exists  a unique semigroup $V(t)$ such that  
\begin{equation}\label{nonlinear}
		\lim_{n\to\infty}\|V_n(t)f_n-V(t)f\|=0
			\end{equation}
		and 
  \begin{equation}\label{eqn_RV}
  \lim_{m\to\infty}\|R(t/m)^m f-V(t)f\|=0.
  \end{equation}
			\item \label{item_Large deviation principle}(Large deviation principle) $X_n(t)$ satisfies the large deviation principle with good rate function $I$ given by
			\begin{equation}\label{eqn_rate_function_1}
			I(x)=I_0(x(t_0))+\sup_{k\in \mathbb{N}}\sup_{0=t_0<t_1<\cdots<t_k<\infty}\sum_{i=0}^{k}I^V_{t_{i+1}-t_{i}}(x(t_{i+1})|x(t_{i})),
			\end{equation}
			where for $\Delta t=t_{i+1}-t_{i}>0$ and $x(t_{i+1}),x(t_{i})\in M$, the conditional rate functions $I^V_{\Delta t}(x(t_{i+1})~ |~ x(t_{i}))$ are 
			\begin{equation*}
				I^V_{\Delta t}(x(t_{i+1})~|~x(t_{i}))=\sup_{f\in C_b(M)}[f(x(t_{i+1}))-V(\Delta t)f(x(t_{i}))].
			\end{equation*} 
			\item \label{item_Action-integral representation of the rate function} (Action-integral representation of the rate function) 
  The rate function $I$ of \eqref{eqn_rate_function_1} together with $V(t)=\mathbf{V}(t)$ is also  given in Action-integral representation,
\begin{equation}\label{eqn_AC}
I(\gamma)=
\begin{cases}
I_0(\gamma(0))+\int^\infty_0\mathcal{L}(\gamma(s),\dot{\gamma}(s) )\text{d}s,&\mbox{if~$\gamma\in \mathcal{AC}(M)$,}\\
					\infty,&\mbox{otherwise.}
				\end{cases}
			\end{equation}
		\end{enumerate}
	\end{proposition}
As a necessary supplement to \Cref{proposition_bridge} \ref{item_e}, we summarize it as the following theorem.
\begin{theorem}[existence of viscosity solutions]\label{thm_viscosity_solution}
	Let $\lambda> 0$ and $h \in C_b(M)$, the resolvent $\mathbf{R}(\lambda)h$ in \Cref{proposition_bridge} \ref{item_e} is a viscosity solution to the Hamilton-Jacobi-Bellman equations $f-\lambda \mathbf{H}f=h$, where $\mathbf{H}$ is defined in \eqref{eqn_bfHH}.
\end{theorem}
We are not urgent to give the proof of \Cref{thm_viscosity_solution}, will prove it during the proof of \Cref{thm_LDPmanifold} in \Cref{sec:variational resolvent is a viscosity solution}. This is one of our main contributions.
\par
Last but not least, according to the strategy, we will have various propositions below to prove \Cref{thm_LDPmanifold}. To visually illustrate the relationship between various theorems and propositions, we present it through a \Cref{Fig:The_strategy_of proof_Thm}.
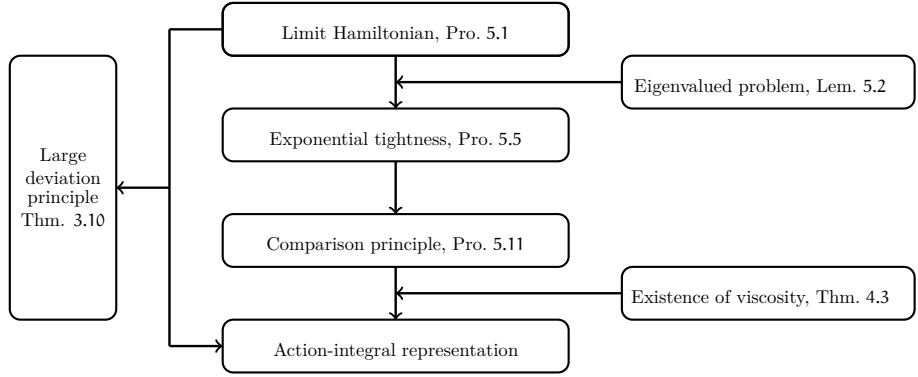
\begin{figure}[h]
    \centering
  \centering
	\begin{tikzpicture}[scale=0.7]
 \draw[thick, rounded corners=4pt] (0,0) rectangle (6.5,1);
 \node at (3.25,0.4) {\scalebox{0.7}{Action-integral representation}};
\draw[<-, line width=0.9pt] (3.25,1) -- (3.25,2);

 \draw[thick, rounded corners=4pt] (0,2) rectangle (6.5,3);
 \draw[<-, line width=0.9pt] (3.25,3) -- (3.25,4);
\node at (3.25,2.4) {\scalebox{0.7}{Comparison principle, $\text{Pro.}~\ref{pro_compa_prin}$}};

\draw[thick, rounded corners=4pt] (0,4) rectangle (6.5,5);
\draw[<-, line width=0.9pt] (3.25,5) -- (3.25,6);
\node at (3.25,4.4) {\scalebox{0.7}{Exponential tightness, $\text{Pro.}~\ref{pro_expo}$}};

\draw[thick, rounded corners=4pt] (0,6) rectangle (6.5,7);
\node at (3.25,6.4) {\scalebox{0.7}{Limit Hamiltonian, $\text{Pro.}~\ref{pro_convergence_of_operator}$}};

 \draw[thick, rounded corners=4pt] (0,6) rectangle (6.5,7);
\node at (-3,4.1) {\scalebox{0.7}{Large}};
\node at (-3,3.7) {\scalebox{0.7}{deviation}};
\node at (-3,3.3) {\scalebox{0.7}{principle}};
\node at (-3,2.9) {\scalebox{0.7}{Thm. $\ref{thm_LDPmanifold}$}};
 \draw[thick, rounded corners=4pt] (-4,6) rectangle (-2,1);
  \draw[->, line width=0.9pt] (-1,0.5) -- (0,0.5);
\draw[-, line width=0.9pt] (-1,0.5) -- (-1,6.5);
\draw[-, line width=0.9pt] (-1,6.5) -- (0,6.5);
  \draw[<-, line width=0.9pt] (-2,3.5) -- (-1,3.5);

\draw[thick, rounded corners=4pt] (7.5,2) rectangle (13,1);
\draw[->, line width=0.9pt] (7.5,1.5) -- (3.25,1.5);
\node at (10.1,5.4) {\scalebox{0.7}{Eigenvalued problem, $\text{Lem.}~\ref{lem_eigen}$}};
  \draw[thick, rounded corners=4pt] (7.5,6) rectangle (13,5);
\draw[->, line width=0.9pt] (7.5,5.5) -- (3.25,5.5);
\node at (10.1,1.4) {\scalebox{0.7}{Existence of viscosity, $\text{Thm.}~\ref{thm_viscosity_solution}$}};
 \end{tikzpicture}
 \caption{The strategy of proof \Cref{thm_LDPmanifold}}	\label{Fig:The_strategy_of proof_Thm}
\end{figure}

\section{The proof of \Cref{thm_LDPmanifold}}\label{sec:the_proof_of_main_theorem}
\subsection{Identification of a single valued Hamiltonian}
In this section, due to existence of switching process $\Lambda_n$ in \eqref{eqn_1goalSDE}, it is not obvious to identify a single valued Hamiltonian on Riemannian manifold. We start to get a multivalued convergent Hamiltonian. 
\subsubsection{Multivalued convergent Hamiltonian}
For any $(x,i)\in O(M) \times S$, let $f(\cdot,\cdot)\in  C^2_c(M\times S;\mathbb{R}^+)$ be the family of all nonnegative functions which are twice differentiable in the spatial variable with support. We consider the solution  $(X_n(t),\Lambda_n(t))$ of the system \eqref{eqn_1goalSDE} and \eqref{eqn_1goalswitch} with the generator $A_n^M$:
\begin{align}\label{eqn_ANM}
   A_n^Mf(x,i)
   =\frac{1}{2n}\Delta_Mf(x,i)+b(x,i)\mathrm{d}f(x,i)+n\sum_{i\in S}q_{ij}(x)(f(x,j)-f(x,i)).
\end{align}

Now we turn to give a multivalued convergent Hamiltonian by the generator $A^M_n$ and the relation $H_n=1/ne^{-nf}A_ne^{nf}$.
\begin{proposition}
    [Multivalued convergent Hamiltonian]\label{pro_convergence_of_operator}
	 Let $(X_n(t),\Lambda_n(t))$ be a Markov process on $M\times S$ with generator $A^M_n$ in \eqref{eqn_ANM}. For $H_n$ in \eqref{eqn_HVN}, then there exists a multivalued limit operator $H$ such that $\|H_n-H\|\to0$ as $n\to \infty$ and $H$ is given by
		\begin{equation}\label{eqn_multivale_H}
		H:=\Big\{(f,H_{f,\phi})\Big|f\in C_b(M), H_{f,\phi}\in C_b(M\times S)~\mbox{and}~\phi\in C^2(M\times S)\Big\},
	\end{equation}
 where 
 \begin{equation}\label{eqn_HFPHi}
     H_{f,\phi}(x,i)=b(x,i)\mathrm{d} f(x)+\frac{1}{2}|\mathrm{d}f(x)|^2+\sum_{j\in S}q_{ij}(x)[e^{\phi(x,j)-\phi(x,i)}-1].
 \end{equation}
\end{proposition}
\begin{proof}
By the generator $A^M_n$ in \eqref{eqn_ANM}, for $\mathrm{e}^{nf}\in \mathcal{D}(A^M_n)$ we get a nonlinear generator
\begin{equation}\label{eqn_HnF}
\begin{split}
   H_nf(x,i)&=\frac{1}{n}e^{-nf}A_n^Me^{nf}(x,i)\\
    &= b(x,i)\mathrm{d} f(x,i)+\frac{1}{2}|\mathrm{d}  f(x,i)|^2+\frac{1}{2n}\Delta_Mf(x,i)\\
    &\quad+\sum_{j\in S}q_{ij}(x)[e^{n(f(x,j)-f(x,i))}-1].
\end{split}
\end{equation}
However, when $n\to \infty$, \eqref{eqn_HnF} is not convergent due to the divergence of the third term. To proceed, instead of using $f$ in \eqref{eqn_HnF}, we take a sequence 
\begin{equation*}
    f_n(x,i)=f(x)+\frac{1}{n}\phi(x,i),~~~\forall~f\in C^\infty(M)~\mbox{and}~\phi\in C^2(M\times S).
\end{equation*}
Obviously, $\mathrm{d}f_n(x,i)=\mathrm{d}f(x)+\frac{1}{n}\mathrm{d}\phi(x,i)$. This, together with \eqref{eqn_HnF}, we have
\begin{align*}
    H_nf_n(x,i)
    &=b(x,i) \left(\mathrm{d}f(x)+\frac{1}{n}\mathrm{d}\phi(x,i)\right)+\frac{1}{2} \left|\mathrm{d}f(x)+\frac{1}{n}\mathrm{d}\phi(x,i)\right|^2\\
    &\quad+\frac{1}{2n}\Delta_M\left(f(x)+\frac{1}{n}\phi(x,i)\right)+\sum_{j\in S}q_{ij}(x)[e^{\phi(x,j)-\phi(x,i)}-1].
\end{align*}
Let $n\to \infty$ in above equality, under uniform topology we obtain the limit
\begin{equation*}
\begin{split}
    &H_{f,\phi}(x,i)\\
    &=b(x,i) \mathrm{d} f(x)+\frac{1}{2}
    \left|\mathrm{d}  f(x)\right|^2+\sum_{j\in S}q_{ij}(x)[e^{\phi(x,j)-\phi(x,i)}-1]\\
&=:H_l(x,i)+H_q(x)+\mathrm{e}^{-\phi(x,i)}R_x\mathrm{e}^{\phi(x,i)}, 
\end{split}
\end{equation*}
where the linear part is
\begin{equation*}
    H_l(x,i)=b(x,i)\text{d} f(x),
\end{equation*}
the nonlinear part is
\begin{equation}\label{eqn_Hq}
   H_q(x)=\frac{1}{2}
   \left|\mathrm{d}  f(x)\right|^2,
\end{equation}
and the switching part is
\begin{equation*}
  R_x \mathrm{e}^{\phi(x,i)}=\sum_{j\in S}q_{ij}(x)[e^{\phi(x,i)}-e^{\phi(x,j)}].
\end{equation*}
The proof is completed.
\end{proof}
\subsubsection{Single valued Hamiltonian}
Based on the preparation of the above subsection, our next task is to solve a principal-eigenvalue problem to obtain a single valued Hamiltonian.
\par
Denote
\begin{equation}\label{eqn_BXI}
    B_{x,\mathrm{d}f(x)}(i)=H_l(x,i)+H_q(x).
\end{equation}

From \eqref{eqn_HFPHi}, one has
\begin{equation*}
	H_{f,\phi}(x,i)=B_{x,\mathrm{d}f(x)}(i)+\mathrm{e}^{-\phi(x,i)}R_x\mathrm{e}^{\phi(x,i)}.
\end{equation*}

One approach, to obtain a single valued Hamiltonian from the multivalue Hamiltonian $H$, is that the principal-eigenvalue problem holds. That is, we try to select eigenvalue functions $\phi(x,i)$ such that
\begin{equation}\label{eqn_gfphi}
	g_{f,\phi}(x):=B_{x,\mathrm{d}f(x)}(i)+\mathrm{e}^{-\phi(x,i)}R_x\mathrm{e}^{\phi(x,i)}.
\end{equation}
Note that $g_{f,\phi}$ is independent of $i$. 
By multiplying $\mathrm{e}^{\phi(x, i)}$ on two sides of \eqref{eqn_gfphi} with given point $x$, we have a new form
\begin{equation*}
    g_{f,\phi}(x)\mathrm{e}^{\phi(x, i)}=(B_{x,\mathrm{d}f(x)}(i)+R_x)\mathrm{e}^{\phi(x,i)}.
\end{equation*}
\par
Based on the above short analysis, we build a principal-eigenvalue problem and get a single valued Hamiltonian which will often used later.
\begin{proposition}[Principal-eigenvalue problem]\label{lem_eigen} Let Assumptions \ref{asm_conservative_manifold} and \cref{asm_conti_manifold} be satisfied. 
	For each $(x,\mathrm{d}f(x))$, there exist $\bar{\phi}>0$ and a unique eigenvalue $\tau(x,\mathrm{d}f(x))\in \mathbb{R}$ such that
\begin{equation}\label{eigenvalue}
	Q_{x,p}\bar{\phi}(i)=\tau(x,\mathrm{d}f(x))\bar{\phi}(i),
\end{equation}
with the operator $Q_{x,p}=B_{x,p}+R_x$. In detail, $\tau(x,\mathrm{d}f(x))$ is given by
\begin{equation}\label{eqn_eigenvalue}
\begin{split}
	\tau(x,\mathrm{d}f(x))&=\inf_{\pi\in \mathcal{P}(S)}\sup_{g>0}\int_M\frac{-Q_{x,p}g(z)}{g(z)}\pi(\text{d}z)\\
	&=-\sup_{\pi\in\mathcal{P}(S)}\inf_{g>0}\int_M\frac{Q_{x,p}g(z)}{g(z)}\pi(\text{d}z).
\end{split}
\end{equation}
Moreover, since 
\begin{equation}\label{eqn:eigenvalue}
	\tau(x,\mathrm{d}f(x))=-\mathcal{H}(x,\mathrm{d}f(x)),
\end{equation}
as a by-product, combining \eqref{eqn_eigenvalue} and \eqref{eqn:eigenvalue}, one has a single valued Hamiltonian 
\begin{equation}\label{eqn_Hxp}
\begin{split}
	\mathbf{H}(x):=\mathcal{H}(x,\mathrm{d}f(x))&=\sup_{\pi\in\mathcal{P}(S)}\inf_{g>0}\int_M\frac{Q_{x,\mathrm{d}f(x)}g(z)}{g(z)}\pi(\text{d}z)\\
	&=\sup_{\pi\in\mathcal{P}(S)}\big\{\int_M B_{x,\mathrm{d}f(x)}(z)\pi(\text{d}z)-\mathcal{I}(x,\pi)\big\},
\end{split}
\end{equation}
where
\begin{equation}\label{eqn_I(x,pi)}
	\mathcal{I}(x,\pi)=-\inf_{g>0}\int_M\frac{R_xg(z)}{g(z)}\pi(\text{d}z).
\end{equation}
\end{proposition}

\begin{proof}
Using Assumptions \ref{asm_conservative_manifold} and \ref{asm_conti_manifold}, from the Perron-Frobenius theorem in \cite{DV1975}, we can obtain there exists a unique eigenvalue with associated eigenfunction which have the represent \eqref{eqn_eigenvalue}. 
\end{proof}

\subsection{Exponential tightness on Riemannian manifold}
In this section, the key step, in obtaining exponential tightness on a Riemannian manifold, is to find a good containment function that can limit our analysis to a compact set.
\par
We now turn to give the definition of good containment functions, which is of importance to get the exponential tightness and comparison principle.
\begin{definition}[Good containment function]
	We say that $\Upsilon:M\to \mathbb{R}$ is a good containment function (for $H$) if 
	\begin{enumerate}[$(\Upsilon a)$]
		\item $\Upsilon\geq 0$ and there exists a point $x_0$ such that $\Upsilon(x_0)=0$,
		\item $\Upsilon$ is twice continuously differentiable,
		\item for every $c\geq 0$, the set $\{x\in M ~|~ \Upsilon(x)\leq c\}$ is compact,
		\item we have $\sup_x\mathcal{H}(x,\mathrm{d}\Upsilon(x))<\infty$.
	\end{enumerate}
\end{definition}

\begin{lemma}\label{lem_fr}
    Let $M$ be a complete Riemannian manifold of dimension $d$. Fix $x_0\in M$ and define $r(x)=d(x,x_0)$. Under \Cref{ass_b_linear_growth}, there 
    exists an nonnegative smooth function $f\in C^\infty(M)$ satisfying   
\begin{equation}\label{eqn_fand_r}
        |f-r|\leq 1
    \end{equation}
    and
\begin{equation}\label{eqn_uniform_bound}
        |\mathrm{d}f|\leq 2
    \end{equation}
    such that 
\begin{equation}\label{eqn_Upsilon}
\Upsilon(x)=\frac{1}{2}\log(1+f^2(x)) 
\end{equation}
is a good containment function.
\end{lemma}
\begin{proof}
    Clearly $\Upsilon\geq 0$ and $\Upsilon(x_0)=0$, and $\Upsilon\in C^\infty(M)$.

    Now fix $c\geq 0$. By the continuity of $\Upsilon$, the set $\{x\in M~|~\Upsilon(x)\leq c\}$ is closed. By definition the set is bounded, and as $M$ is a complete finite dimensional Riemannian manifold, also compact.

    Note that for all $x\in M$,
    \begin{equation}\label{eqn_dUps}
\text{d}\Upsilon(x)=\frac{f(x)}{1+f^2(x)}\mathrm{d}f(x).
    \end{equation}

   This, together with  \Cref{ass_b_linear_growth}, \eqref{eqn_uniform_bound} and \eqref{eqn_fand_r}, for $z\in S$, we first estimate that
\begin{equation}\label{eqn_useoneside}
 \begin{split}
   b(x,z)  \mathrm{d}\Upsilon(x)&=b(x,z) \mathrm{d}f(x)\frac{f(x)}{1+f^2(x)} \\
    &\leq |b(x,z)| \cdot |\mathrm{d}f(x)|\cdot\frac{f(x)}{1+f^2(x)} \\&\leq 2C_b(1+r(x))\frac{f(x)}{1+f^2(x)} \\
    &\leq 2C_b(2+f(x))\frac{f(x)}{1+f^2(x)}.
 \end{split}
 \end{equation}
 Hence, $\sup_{x,z}b(x,z)\dd \Upsilon(x)<\infty$.
    Now recall the Hamiltonian $\mathcal{H}$ in \eqref{eqn_Hxp}, from \eqref{eqn_useoneside}, we obtain
\begin{align*}
    \mathcal{H}(x,\mathrm{d}\Upsilon(x))&=\sup_{\pi\in\mathcal{P}(S)}\big\{\int_M B_{x,\mathrm{d}\Upsilon(x)}(z)\pi(\text{d}z)-\mathcal{I}(x,\pi)\big\}\\
    &\leq \int_M B_{x,\mathrm{d}\Upsilon(x)}(z)\pi(\text{d}z)\\
     &=\int_M \left( b(x,z)\mathrm{d}\Upsilon(x)+\frac{1}{2}| \mathrm{d}\Upsilon(x)|^2\right)\pi(\mathrm{d}z)\\
     &\leq C\int_M \left(\frac{f(x)}{1+f^2(x)}+\frac{f^2(x)}{1+f^2(x)}+\frac{f^2(x)}{(1+f^2(x))^2}\right)\pi(\mathrm{d}z),
\end{align*}
where the first inequality uses the definition of supremum and $\mathcal{I}(x,\pi)$ is nonnegative.
We conclude that $\sup_x \mathcal{H}(x,\mathrm{d}\Upsilon(x))<\infty$, which implies that $\Upsilon$ is a good containment function.
\end{proof}
Applying the good containment function \eqref{eqn_Upsilon}, we proceed to consider the exponential tightness of the system $(X_n(t),\Lambda_n(t))$. To do this, we note that by \cite[Corollary 4.19]{FK2006} it sufﬁces to establish the exponential compact containment condition.
\begin{proposition}[Exponential compact containment condition]\label{pro_expo}
 Let $(X_n(t),\Lambda_n(t))$ be a Markov process corresponding to $A^M_n$, and $\Upsilon(x)$ be a good containment function in \eqref{eqn_Upsilon}. 
    Suppose that the sequence $(X_n(t),\Lambda_n(t))$ is exponentially tight with speed $n$.
    Then the sequence $(X_n(t),\Lambda_n(t))$ satisfies the exponential compact containment condition with speed $n$: for every $T > 0$ and $a \geq  0$, there exists a compact set $K_{a,T}\subseteq M$ such that
    \begin{equation*}
        \limsup_{n\to \infty} \frac{1}{n}\log \mathbb{P}[(X_n(t),\Lambda_n(t))\notin K_{a,T}\times S ~\mbox{for~some}~t\leq T]\leq -a.
    \end{equation*}
\end{proposition}

\begin{proof}
   The proof is similar to our work \cite[Proposition 5.4]{HKX2023}. 
   The different parts of the proof are that $\Upsilon$ is related to a smooth function $f\in C^\infty(M)$ based on abstract properties of $\Upsilon$ verified in \Cref{lem_fr} since we work on a complete Riemannian manifold.
\end{proof}

\subsection{Comparison principle on Riemannian manifold}
\subsubsection{The smooth distance function to prove comparison principle}
Finding a smooth distance function on a complete Riemannian manifold to prove the comparison principle is the technology parts to prove \Cref{thm_LDPmanifold}. Inspired by \cite{FK2006}, we consider a Cauchy problem about $H_q$ defined in \eqref{eqn_Hq},
\begin{equation*}
    \frac{\partial}{\partial t}u(t,x_0)=H_qu(t,x_0),~~~~u(0,x_0)=h(x_0).
\end{equation*}
The solution $u(t,x_0)$ is given by the Hopf-Lax formula \cite[Proposition 9.4.1]{BGL2014MR3155209},
\begin{equation}\label{eqn_solution_u}
\begin{split}
u(t,x_0)&=\sup\left\{h(x(t))-\frac{1}{2}\int^t_0 g_{x(s)}(\dot{x}(s),\dot{x}(s))\mathrm{d}s\right\}\\
    &=\sup_{y_0\in M}\left\{h(y_0)-\frac{d^2(x_0,y_0)}{2t}\right\}.
\end{split}
\end{equation}
Then, we obtain a  distance function from \eqref{eqn_solution_u}, 
\begin{equation}\label{eqn_distance_function}
    d^2(x_0,y_0)=\inf\left\{\int^1_0 |\dot{x}(s)|^2\mathrm{d}s:x(t)\in M; 0\leq t\leq 1, x(0)=x_0,x(1)=y_0\right\}.
\end{equation}
 We continue to give the differential property of the distance function, the proof was shown in \cite[Appendix C.1]{KRV2019}.
\begin{lemma}\label{lem_dddd}
    Let $x$, $y\in M$ and assume that $x\notin \mbox{cut}(y)$ (or equivalently, $y\notin \mbox{cut}(x)$), where $\mbox{cut}(\cdot)$ is a cut-locus. Then for all $V\in T_yM$ we have
    \begin{equation*}
		\mathrm{d}_y(d^2(x,\cdot))(y)(V)=2\langle \dot{\gamma}(1),~V
\rangle_{g(y)},    
\end{equation*}
where $\gamma: [0,1]\to M$ is the unique geodesic of minimal length connecting $x$ and $y$. Consequently, we obtain
\begin{equation}\label{eqn_tauxy}
   \tau_{x,y} \mathrm{d}_x(d^2(\cdot,y))(x)=-\mathrm{d}_y(d^2(x,\cdot))(y).
\end{equation}
\end{lemma}
Noting that \eqref{eqn_tauxy} shows the relation of the derivates of distance function about variable $x$ and $y$ on Riemannian maniford, it is vital when we prove the comparison principle.
\par
Our approach to proving the comparison principle is to double variables, as in the classical setting of viscosity solutions in Euclidean spaces, using the distance function as a penalizing function. However, the distance function is not smooth on $M$ because the shortest path between two points (i.e. the geodesic) may not be unique, which was briefly discussed in the \Cref{se_Riemannian_manifold}.
we must modify further penalizations to force the extrema to be attained. To this goal, taking advantage of  the property of the cut-locus, we have a smooth distance function or good penalization function.
\begin{lemma}\label{lem_smooth_function}
    Consider a compact set $K \subseteq M$. Then there is $\delta = \delta_K > 0$ and a smooth function $\Psi = \Psi_K$ with $\Psi(x,~y) \in C^\infty(M^2)\cap C_b(M^2)$ such that $\Psi(x,~y) = \frac{1}{2}d^2(x,~y)$ if $d(x,~y) <\delta $ and $x$, $y\in K$.
\end{lemma}
\begin{proof}
    As $K$ is compact, the continuity of the injectivity radius $i$ implies there exists $R > 0 $ such that $i(K ) > R$. Pick $\Delta  =\frac{1}{3}R$ and let $\Phi: \mathbb{R}\to \mathbb{R}$ be a smooth increasing function such that 
    \begin{equation*}
\Phi(r)=
\begin{cases}
r ,&if ~\mbox{$|r|<\Delta$,}\\
1,&if~\mbox{$|r|>2\Delta$.}
\end{cases}
\end{equation*} Then the function $\Psi(x,y)=\frac{1}{2}\Phi(d^2(x,y))=\frac{1}{2}d^2(x,y)$, when $d(x,y)\leq \Delta$ and $x$, $y\in K$, is as desired.
\end{proof}

Along the ideas of double variable, we continue to give the following lemma with the good penalization function $\Psi(x,y)$ defined in \Cref{lem_smooth_function}.

\begin{lemma}[Lemma A.10 in \cite{CK_2017}]\label{lem_pro3.7}
Let $u$ be bounded and upper semicontinuous, let $v$ be bounded and lower semicontinuous, let $\Psi:M^2\to \mathbb{R}^+$ be good penalization functions in \Cref{lem_smooth_function} and let $\Upsilon$ be a good containment function in \eqref{eqn_Upsilon}. 
\par
Fix $\delta>0$. For every $m>0$ there exist points $x_{\delta,m}$, $y_{\delta,m}\in M$, such that
\begin{equation}\label{eqn:Text_function}
\begin{split}
&\frac{u(x_{\delta,m})}{1-\delta}-\frac{v(y_{\delta,m})}{1+\delta}-m\Psi(x_{\delta,m},y_{\delta,m})-\frac{\delta}{1-\delta}\Upsilon(x_{\delta,m})-\frac{\delta}{1+\delta}\Upsilon(y_{\delta,m})\\
&=\sup_{x,y\in M}\big\{\frac{u(x)}{1-\delta}-\frac{v(y)}{1+\delta}-m\Psi(x,y)-\frac{\delta}{1-\delta}\Upsilon(x)-\frac{\delta}{1+\delta}\Upsilon(y)\big\}.
\end{split}
\end{equation}
Additionally, for every $\delta>0$ we have that
\begin{enumerate}
\item \label{item_iaa} The set $\{x_{\delta,m},~y_{\delta,m} ~|~m >0\}$ is relatively compact in $M$.
\item \label{item_ibb} All limit points of $\{(x_{\delta,m},~y_{\delta,m})\}_{m >0}$ are of the form $(z,~z)$ and for these limit points we have 
\begin{equation*}
 u(z)-v(z)=\sup_{x\in M}u(x)-v(x).   
\end{equation*}
\item \label{item_icc} We have
\begin{equation*}
\lim_{m\to\infty}m\Psi(x_{\delta,m},~y_{\delta,m})=0.
\end{equation*}
\end{enumerate}
\end{lemma}

\subsubsection{The necessary operators for proving comparison principle}
Denote by $C_l^{\infty}(M)$ the set of smooth functions on $M$ that has a lower bound and by $C_u^\infty(M)$ the set of smooth functions on $M$ that has an upper bound. $\Upsilon$ in \eqref{eqn_Upsilon} is a good containment function such that $C_\Upsilon:=\sup_{(x,i)\in M\times S}B_{x,\dd \Upsilon(x)}(i)<\infty$.

We use the multivalued convergent Hamiltonian $H$ to define two new multivalued valued operators for proving the comparison principle in which we include $\Upsilon$ to be able to restrict the analysis to compact sets.
\begin{definition}(Multivalued operators)
\begin{itemize}
\item For $f\in C^{\infty}_l(M)$, $\delta\in(0,1)$ and $\phi\in C^2(M\times S)$. Set
\begin{equation*}
f^{\delta}_1(x):=(1-\delta)f(x)+\delta \Upsilon(x),
\end{equation*}
\begin{equation*}
H^\delta_{1,f,\phi}(x,i):=(1-\delta)H_{f,\phi}(x,i)+\delta C_{\Upsilon},
\end{equation*}
and set
\begin{equation*}
H_1:=\Big\{\left(f^{\delta}_1,H^\delta_{1,f,\phi}\right)\Big | f\in C^{\infty}_l(M),\delta\in(0,1),~\phi\in C^2(M\times S)\Big\}.
\end{equation*}
\item For $f\in C^{\infty}_u(M)$, $\delta\in(0,1)$ and $\phi\in C^2(M\times S)$. Set
\begin{equation*}
f^{\delta}_2(x):=(1+\delta)f(x)-\delta \Upsilon(x),
\end{equation*}
\begin{equation*}
H^\delta_{2,f,\phi}(x,i):=(1+\delta)H_{f,\phi}(x,i)-\delta C_\Upsilon,
\end{equation*}
and set
\begin{equation*}
H_2:=\left\{\left(f^{\delta}_2,H^\delta_{2,f,\phi}\right)\Big | f\in C^{\infty}_u(M),~\delta\in(0,1),~\phi\in C^2(M\times S)\right\}.
\end{equation*}
Noting that $H_{f,\phi}$ in $H_{1,f,\phi}^\delta$ and $H_{2,f,\phi}^\delta$ is obtained in \eqref{eqn_HFPHi}.
\end{itemize}
\end{definition}
We use the single valued Hamiltonian $\mathbf{H}$ to define two new single valued operators. 
\begin{definition}(Single valued operators)
\begin{itemize}
\item For $f\in C^{\infty}_l(M)$ and $\delta \in(0,1)$ set
\begin{equation*}
f^{\delta}_{\dagger}(x):=(1-\delta)f(x)+\delta \Upsilon(x),
\end{equation*}
\begin{equation*}
H^{\delta}_{\dagger,f}(x):=(1-\delta)\mathbf{H}f(x)+\delta C_{\Upsilon},
\end{equation*}
and set
\begin{equation*}
H_{\dagger}:=\left\{\left(f^{\delta}_\dagger,H^{\delta}_{\dagger,f}\right)\Big |f\in C^\infty_l(M),~\delta\in(0,1)\right\}.
\end{equation*}
\item For $f\in C^{\infty}_u(M)$ and $\delta\in(0,1)$ set
\begin{equation*}
f^{\delta}_{\ddagger}(x):=(1+\delta)f(x)-\delta \Upsilon(x),
\end{equation*}
\begin{equation*}
H^{\delta}_{\ddagger,f}(x):=(1+\delta)\mathbf{H}f(x)-\delta C_{\Upsilon},
\end{equation*}
and set
\begin{equation*}
H_{\ddagger}:=\Big\{\left(f^\delta_\ddagger,~H^\delta_{\ddagger,f}\right)\Big |f\in C^\infty_u(M),~\delta\in(0,1)\Big\}.
\end{equation*}
\end{itemize}
Noting that $\mathbf{H}$ in $H_{\dagger,f}^\delta$ and $H_{\ddagger,f}^\delta$ is obtained in \eqref{eqn_Hxp}.
\end{definition}

We collect $H$, $\mathbf{H}$, $H_\dagger$, $H_\ddagger$, $H_1$ and $H_2$ in \Cref{3}, which intuitively provides the proof strategy for the comparison principle in the following subsection.
\begin{figure}[htp]
	\centering
	\begin{tikzpicture}[>=stealth,xscale=0.96,yscale=0.78]
	\node[coordinate] (nw) at (-6.8,2.2) {};
	\node[coordinate] (sw) at (-6.8,-2.2) {};
	\node[coordinate] (se) at (3.5,-2.2) {};
	\node[coordinate] (ne) at (3.5,2.2) {};
	\node[coordinate] (e) at (3.5,0) {};
	\node[coordinate] (w) at (-6.8,0) {};
	\draw[->, very thick]   (-5,0.2)--node[left=1.5mm,above,sloped]{sub}(-3,1);
	\draw[->, very thick]   (-5,-0.2)--node[left=1.5mm,below,sloped]{super}(-3,-1);
	\draw[->, very thick]   (-2.5,1)--node[above=1mm]{sub}(0,1);
	\draw[->, very thick]   (-2.5,-1)--node[below=1mm]{super}(0,-1);
	\draw[<-, very thick]   (0.55,-1)--node[right=1.5mm,below,sloped]{super}(2.5,-0.2);
	\draw[<-, very thick]   (0.55,1)--node[right=1.5mm,above,sloped]{sub}(2.5,0.2);
	\fill[gray, draw=black,rounded corners,fill opacity=0.3] (-0.7,-1.35) rectangle (1.3,1.5);
	\node at (-5.5,0){$\Large \text{ H}$};
 \node at (-6.5,0.28){$ \text{implicit}$};
 \node at (-6.6,-0.3){$ \text{multivalued}$};
 \fill[blue, draw=black,rounded corners,fill opacity=0.3] (-7.6,0.8) rectangle (-5,-0.8);
 \node at (-2.8,1){$\Large \text{ H}_1$};
	\node at (-2.8,-1){$\Large \text{ H}_2$};
	\node at (0.2,-1){$\Large \text{ H}_\ddagger$};
	\node at (0.2,1){$\Large \text{ H}_\dagger$};
	\node at (2.7,0){$\Large \textbf{ H}$};
 \node at (4,0.28){$\text{explicit}$};
  \node at (4.2,-0.3){$ \text{single~valued}$};
  \fill[blue, draw=black,rounded corners,fill opacity=0.3] (2.5,0.8) rectangle (5.3,-0.8);
	\node at (0.3,0){comparison };
	\end{tikzpicture}
	\caption{An arrow connecting an operator $A$ with operator $B$ with subscript `sub’ means that viscosity subsolutions of $f-\lambda Af=h$ are also viscosity subsolutions of $f-\lambda Bf=h$. Similarly, we get the description for arrows with a subscript `super'. The middle gray box around the operators $H_\dagger$ and $H_\ddagger$ indicates that the comparison principle holds for subsolutions of $f-\lambda H_\dagger f=h$ and supersolutions of $f-\lambda H_\ddagger f=h$. The left blue box indicates that $H$ is an implicit and multivalued operator. The right blue box indicates $\mathbf{H}$ is an explicit single valued operator.}
	\label{3}
\end{figure}
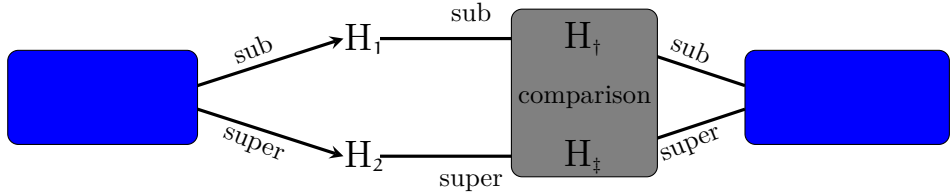
\subsubsection{Main propositions: comparison principle}
Based on the above preparations, we are ready to state the proposition of this subsection.
\begin{proposition}[Comparison principle]\label{pro_compa_prin}
Let Assumptions \ref{ass_b_linear_growth}, \ref{ass_b_one_side}, \ref{asm_conservative_manifold} and \ref{asm_conti_manifold} be satisfied. Let $h_1$, $h_2\in C_b(M)$ and $\lambda>0$. Let $u$ be any subsolution to $f-\lambda H f=h_1$ and let $v$ be any supersolution to $f-\lambda Hf=h_2$. Then we have that
	\begin{equation*}
		\sup_{x\in M}u(x)-v(x)\leq\sup_{x}h_1(x)-h_2(x).
	\end{equation*}
\end{proposition}
 To prove comparison principle \Cref{pro_compa_prin}, we need the following lemmas on Riemannian manifold such that all the implications in \Cref{3} can be established. Those proof is analogous to the comparison principle in \cite{HKX2023} on the Euclidean space.
\begin{lemma}\label{lem_1d}
Fix $\lambda>0$ and $h\in C_b(M)$.
\begin{enumerate}
\item Every subsolution to $f-\lambda H_1f=h$ is also a subsolution to $f-\lambda H_{\dagger}f=h$.
\item Every supersolution to $f-\lambda H_1f=h$ is also a supersolution to $f-\lambda H_{\ddagger}f=h$.
\end{enumerate}
\end{lemma}

\begin{lemma}\label{lem_bfd}
Fix $\lambda>0$ and $h\in C_b(M)$. 
\begin{enumerate}
\item Every subsolution to $f-\lambda \mathbf{H}f=h$ is also a subsolution to $f-\lambda H_\dagger f=h$.
\item Every supersolution to $f-\lambda \mathbf{H}f=h$ is also a supersolution to $f-\lambda H_{\ddagger} f=h$.	
\end{enumerate}
\end{lemma}

\begin{lemma}\label{lem_h1}
Fix $\lambda >0$ and $h\in C_b(M)$.
\begin{enumerate}
\item Every subsolution to $f-\lambda Hf=h$ is also a subsolution to $f-\lambda H_1f=h$.
\item  Every supersolution to $f-\lambda Hf=h$ is also a supersolution to $f-\lambda H_2f=h$.
\end{enumerate}
\end{lemma}

In addition to the lemmas above,  we still need to verify the comparison principle for $f-\lambda H_\dagger f=h_1$ and $f-\lambda H_\ddagger f=h_2$ on $M$ from \Cref{3}.

\begin{lemma}\label{Hd}
Assumptions \ref{assu_varn_manifold}, \ref{ass_b_one_side}, \ref{asm_conservative_manifold} and \ref{asm_conti_manifold} hold. Let $h_1$, $h_2\in C_b(M)$ and $\lambda>0$. Let $u$ be any subsolution to $f-\lambda H_{\dagger}f=h_1$ and let $v$ be any supersolution to $f-\lambda H_{\ddagger}f=h_2$. Then we have
\begin{equation}\label{eqn:results_comparison_principle}
\sup_{x\in M}u(x)-v(x)\leq\sup_{x\in M}h_1(x)-h_2(x).
\end{equation}
\end{lemma}

\begin{proof} 
    For a sub and supersolution $u$ and $v$, $\delta \in (0,1)$ and $m \geq 1$, we follow \eqref{eqn:Text_function} and set
 \begin{equation}\label{eqn:Pdelta_m}
\Phi_{\delta,m}(x,y):=\frac{u(x)}{1-\delta}-\frac{v(y)}{1+\delta}-m\Psi(x,y)-\frac{\delta}{1-\delta}\Upsilon(x)-\frac{\delta}{1+\delta}\Upsilon(y),
\end{equation}
where $\Psi(\cdot,\cdot)$ is a smooth function we have given in \Cref{lem_smooth_function}. Since $\Upsilon$ has compact level sets, there exists $(x_{\delta,m},y_{\delta,m})\in M\times M$ satisfying 
\begin{equation}\label{eqn_Phi_sup}
\Phi_{\delta,m}(x_{\delta,m},~y_{\delta,m})=\sup_{(x,y)\in M\times M}\Phi_{\delta,m}(x,~y).
\end{equation}
In view of \eqref{eqn_Phi_sup}, it follows that $x_{\delta,m}$ is the unique maximizing point for
\begin{align*}
\sup_{x\in M}u(x)-\varphi_1^{\delta,m}(x)=u(x_{\delta,m})-\varphi_1^{\delta,m}(x_{\delta,m})
\end{align*}
where $\varphi_1^{\delta,m}$ is constructed by taking the appropriate remaining terms of \eqref{eqn:Pdelta_m}, with an additional penalization $(1-\delta)d^2(x,x_{\delta,m})$ to turn $x_{\delta,m}$ into the unique optimizer:
\begin{align*}
\varphi_1^{\delta,m}(x):&=-(1-\delta)\Phi_{\delta,m}(x,y_{\delta,m})+u(x)+(1-\delta)d^2(x,x_{\delta,m})\\
&=(1-\delta)\left(-\frac{u(x)}{1-\delta}+\frac{v(y_{\delta,m})}{1+\delta}+m\Psi(x,y_{\delta,m})+\frac{\delta}{1-\delta}\Upsilon(x)+\frac{\delta}{1+\delta}\Upsilon(y_{\delta,m})\right)\\
&\quad+u(x)+(1-\delta)d^2(x,x_{\delta,m})\\
&=(1-\delta)\left(\frac{v(y_{\delta,m})}{1+\delta}+m\Psi(x,y_{\delta,m})+\frac{\delta}{1-\delta}\Upsilon(x)+\frac{\delta}{1+\delta}\Upsilon(y_{\delta,m})\right)\\
&\quad+(1-\delta)d^2(x,x_{\delta,m})\\
&=(1-\delta)\bigg(m\Psi(x,y_{\delta,m})+d^2(x,x_{\delta,m})+\frac{\delta}{1+\delta}\Upsilon(y_{\delta,m})+\frac{v(y_{\delta,m})}{1+\delta}\bigg)+\delta\Upsilon(x).
\end{align*}
Since $u$ is a viscosity subsolution of $f-\lambda H_{\dagger}f=h_1$ we conclude that
\begin{equation}\label{eqn_sub_inequality}
u(x_{\delta,m})-\lambda\left[(1-\delta)\mathcal{H}(x_{\delta,m},p^1_{\delta,m})+\delta C_{\Upsilon}\right]\leq h_1(x_{\delta,m}),
\end{equation}
where
\begin{equation}\label{eqn_P1} p^1_{\delta,m}:=md_x\Psi(\cdot,y_{\delta,m})(x_{\delta,m})=m~ \mathrm{d}_{x}\left(\frac{1}{2} d^2(\cdot,y_{\delta,m})\right)(x_{\delta,m}).
\end{equation}
Similarly, we obtain that that $y_{\delta,m}$ it the unique optimizer for
\begin{equation*}
\inf_{y\in M}v(x)-\varphi_2^{\delta,m}(y)=v(y_{\delta,m})-\varphi_2^{\delta,m}(y_{\delta,m}),
\end{equation*}
where
\begin{align*}
\varphi^{\delta,m}_2(y):&=(1+\delta)\left(-m\Psi(x_{\delta,m},y)-d^2(y,y_{\delta,m})-\frac{\delta}{1-\delta}\Upsilon(x_{\delta,m})+\frac{u(x_{\delta,m})}{1-\delta}\right)-\delta\Upsilon(y).
\end{align*}
As $v$ is a viscosity supersolution of $f-\lambda H_{\ddagger}f=h_2$, we know then that
\begin{equation}\label{eqn_super_inequality}
v(x_{\delta,m})-\lambda\left[(1+\delta)\mathcal{H}(y_{\delta,m},p^2 _{\delta,m})-\delta C_{\Upsilon}\right]\geq h_2(y_{\delta,m}),
\end{equation}
where
\begin{equation}\label{eqn_P2}
p^2_{\delta,m}:=-m~\mathrm{d}_y \left(\frac{1}{2}d^2(x_{\delta,m},\cdot)\right)(y_{\delta,m}).
\end{equation}
By \cref{item_icc} of \Cref{lem_pro3.7}, we have
\begin{equation}\label{eq_md}
\lim_{m\to\infty}m\Psi(x_{\delta,m},y_{\delta,m})=0.
\end{equation}
Taking \eqref{eqn_sub_inequality}, \eqref{eqn_super_inequality} and \eqref{eq_md} into account, we obtain that
\begin{align}
\sup_{x\in M}u(x)-v(x)&\leq \liminf_{\delta\to0}\liminf_{m\to\infty}\left(\frac{u(x_{\delta,m})}{1-\delta}-\frac{v(y_{\delta,m})}{1+\delta}\right)\notag\\
&\leq \liminf_{\delta\to0}\liminf_{m\to\infty}\big\{\frac{h_1(x_{\delta,m})}{1-\delta}-\frac{h_2(y_{\delta,m})}{1+\delta}\label{eqn_k1}\\
&\quad+\frac{\delta}{1-\delta}C_{\Upsilon}+\frac{\delta}{1+\delta}C_{\Upsilon}\label{eqn_k2}\\
&\quad+ \lambda\big(\mathcal{H}(x_{\delta,m},p^1_{\delta,m})-\mathcal{H}(y_{\delta,m},p^2_{\delta,m})\big)\big\},
\end{align}
where in the first inequality we use \eqref{eqn_Phi_sup} and  drop the nonnegative functions  $d^2(\cdot,\cdot)$ and $\Upsilon(\cdot)$.
\par
The term \eqref{eqn_k2} vanishes as $\delta \to 0$.
For the term \eqref{eqn_k1}, the sequence $(x_{\delta,m},y_{\delta,m})$ takes its values in a compact set and, hence, admits converging subsequences as $m\to \infty$. 
By (b) of \Cref{lem_pro3.7}, these subsequences converge to points of the form $(x,x)$. 
Hence, by the above analysis, we get 
\begin{align*}
	\sup_{x\in M}u(x)-v(x)&\leq \lambda\liminf_{\delta\to0}\liminf_{m\to\infty} \big(\mathcal{H}(x_{\delta,m},p^1_{\delta,m})-\mathcal{H}(y_{\delta,m},p^2_{\delta,m})\big)\\
	&\quad+\sup_{x\in M}h_1(x)-h_2(x).
\end{align*}
It follows that the comparison principle holds for $f-\lambda H_\dagger f=h_1$ and $f-\lambda H_\ddagger f=h_2$ whenever for any $\delta>0$
\begin{equation}\label{eqn_H-H}
 	\liminf_{m\to\infty} \big(\mathcal{H}(x_{\delta,m},p^1_{\delta,m})-\mathcal{H}(y_{\delta,m},p^2_{\delta,m})\big)\leq0.
\end{equation}
\par
To that end, recall $\mathcal{H}(x,\mathrm{d}f(x))$ in \eqref{eqn_Hxp}:
\begin{equation*}
\mathcal{H}(x,\mathrm{d}f(x))=\sup_{\pi\in\mathcal{P}(S)}\big\{\int_M B_{x,\dd f(x)}(z)\pi(\text{d}z)-\mathcal{I}(x,\pi)\big\},
\end{equation*}
where $\pi\mapsto \int_M B_{x,\dd f(x)}(z)\pi(\mathrm{d}z)$ is bounded and continuous, and $\mathcal{I}(x,\cdot)$ has compact sub-level sets in $\mathcal{P}(S)$. Thus, there exists an optimizer $\pi_{\delta,m}\in\mathcal{P}(S)$ such that 
\begin{equation}\label{eqn_H1}
	\mathcal{H}(x_{\delta,m},p^1_{\delta,m})=\int_M B_{x_{\delta,m},p^1_{\delta,m}}(z)\pi_{\delta,m}(\text{d}z)-\mathcal{I}(x_{\delta,m},\pi_{\delta,m})
\end{equation}
and 
\begin{equation}\label{eqn_H2}
	\mathcal{H}(y_{\delta,m},p^2_{\delta,m})\leq \int_M B_{y_{\delta,m},p^2_{\delta,m}}(z)\pi_{\delta,m}(\text{d}z)-\mathcal{I}(y_{\delta,m},\pi_{\delta,m}).
\end{equation}
Combining \eqref{eqn_H1} and \eqref{eqn_H2}, we obtain 
\begin{align}
	\mathcal{H}&(x_{\delta,m},p^1_{\delta,m})-\mathcal{H}(y_{\delta,m},p^2_{\delta,m})\notag\\
	&\leq \int_M\big(B_{x_{\delta,m},p^1_{\delta,m}}(z)-B_{y_{\delta,m},p^2_{\delta,m}}(z)\big)\pi_{\delta,m}(\text{d}z)\label{eqn_B_B}\\
	&\quad+\mathcal{I}(y_{\delta,m},\pi_{\delta,m})-\mathcal{I}(x_{\delta,m},\pi_{\delta,m})\label{eqn_I-I}.
\end{align}
It is enough to prove that \eqref{eqn_B_B} and \eqref{eqn_I-I} are sufficiently small. For \eqref{eqn_B_B}, by calculating the difference of integrand  $B_{x,p}$ in detail, for any $z \in S$, from \eqref{eqn_BXI}, \eqref{eqn_P1} and \eqref{eqn_P2} one has
\begin{equation}\label{eqn:B-B}
\begin{split}
&B_{x_{\delta,m},p^1_{\delta,m}}(z)-B_{y_{\delta,m},p^2_{\delta,m}}(z)\\&
=m\mathrm{d}_{x}\left(\frac{1}{2} d^2(\cdot,y_{\delta,m})\right)(x_{\delta,m})b(x_{\delta,m},z)+\frac{1}{2}\left| m~ \mathrm{d}_{x}\left(\frac{1}{2} d^2(\cdot,y_{\delta,m})\right)(x_{\delta,m})  \right|^2\\
&-\left[ -m~\mathrm{d}_y \left(\frac{1}{2}d^2(x_{\delta,m},\cdot)\right)(y_{\delta,m})b(y_{\delta,m},z)+\frac{1}{2}\left| -m~\mathrm{d}_y \left(\frac{1}{2}d^2(x_{\delta,m},\cdot)\right)(y_{\delta,m})\right|^2 \right]\\
&=m\mathrm{d}_{x}\left(\frac{1}{2} d^2(\cdot,y_{\delta,m})\right)(x_{\delta,m})b(x_{\delta,m},z)+m~\mathrm{d}_y \left(\frac{1}{2}d^2(x_{\delta,m},\cdot)\right)(y_{\delta,m})b(y_{\delta,m},z) \\
&+\frac{m^2}{2}\left(\left|\mathrm{d}_{x}\left(\frac{1}{2} d^2(\cdot,y_{\delta,m})\right)(x_{\delta,m}) \right|^2-\left|- \mathrm{d}_y \left(\frac{1}{2}d^2(x_{\delta,m},\cdot)\right)(y_{\delta,m}) \right|^2\right)\\
&=m\mathrm{d}_{x}\left(\frac{1}{2} d^2(\cdot,y_{\delta,m})\right)(x_{\delta,m})b(x_{\delta,m},z)+m~\mathrm{d}_y \left(\frac{1}{2}d^2(x_{\delta,m},\cdot)\right)(y_{\delta,m})b(y_{\delta,m},z)\\
&
\leq Cd^2(x_{\delta,m},y_{\delta,m}),
\end{split}
\end{equation}
where in the last inequality, we use \Cref{ass_b_one_side}. Noting that
the last term in line 5 vanishes. This is happened because, fix $\delta>0$, there is a compact $K^\delta\subseteq M$ such that $\{x_{m,\delta},~y_{m,\delta}~|~m>0\}$ is contained in $K^\delta$ by \cref{item_iaa} of \Cref{lem_pro3.7}. By the continuity of the injectivity radius and the compactness of $K^\delta$, we can find a $\Delta>0$ such that $i(K^\delta)\geq \Delta >0$. Then there exists a unique geodesic of minimal length connecting $x_{m,\delta}$ and $y_{m,\delta}$. 
Furthermore, by \Cref{lem_dddd} we have
\begin{equation}\label{eqn_staud}
 \mathrm{d}_xd^2(\cdot,y_{m,\delta})(x_{m,\delta})=  -\tau_{x_{m,\delta},~y_{m,\delta}} \mathrm{d}_yd^2(x_{m,\delta},~\cdot)(y_{m,\delta}),
\end{equation}
where $\tau_{x_{m,\delta},y_{m,\delta}}$ denotes parallel transport along the unique geodesic of minimal length connecting $x_{m,\delta}$ and $y_{m,\delta}$. As parallel transport is an isometry, we find
\begin{equation*}
\left|\mathrm{d}_{x}\left(\frac{1}{2} d^2(\cdot,y_{\delta,m})\right)(x_{\delta,m}) \right|_{g(x)}^2=\left|- \mathrm{d}_y \left(\frac{1}{2}d^2(x_{\delta,m},\cdot)\right)(y_{\delta,m}) \right|^2_{g(y)}
\end{equation*}
\par
Hence, \eqref{eqn_B_B} is sufficiently small, as $m\to \infty$,  using \eqref{eq_md}  and \eqref{eqn:B-B}. 
We obtain that \eqref{eqn_I-I} is sufficiently small, utilize the equi-continuity of $\mathcal{I}(\cdot,\pi)$ established in \Cref{lem_Lip} below for the spatial variable. This finishes the proof of \eqref{eqn_H-H} and the comparison principle for $H_\dagger$ and $H_\ddagger$.
 \end{proof}

Here, we state the equi-continuity of $\mathcal{I}(\cdot,\pi)$ to finish the proof of the comparison principle of $H_\dagger$ and $H_\ddagger$ in \Cref{Hd}. The proof was shown on Lemma 6.11 of our work \cite{HKX2023} and extended to complete Riemannian manifold.
\begin{lemma}\label{lem_Lip}
Let \Cref{asm_conti_manifold} be satisfied. Recall \eqref{eqn_I(x,pi)}:
\begin{align*}
\mathcal{I}(x,\pi)&=-\inf_{g>0}\int_{M}\frac{R_{x}g(z)}{g(z)}\pi(\mathrm{d}z).
\end{align*}
For any compact set $K\subseteq M$ and for all $\pi\in \mathcal{P}(S)$, then $\{x\mapsto J(x,\pi)\}_{x\in K}$ is equi-continuous.
\end{lemma}

\subsubsection{Proof of \Cref{pro_compa_prin}}
We now prove \Cref{pro_compa_prin}; that is, the verification of the comparison principle for the Hamilton-Jacobi equations $f-\lambda Hf=h$. The proof get from \Cref{lem_1d}, \Cref{lem_bfd}, \Cref{lem_h1} and \Cref{Hd}. 
Furthermore, we first obtain \Cref{4} via adding these lemmas in \Cref{3} as below for an easy understanding of the proof strategy of \Cref{pro_compa_prin}. 
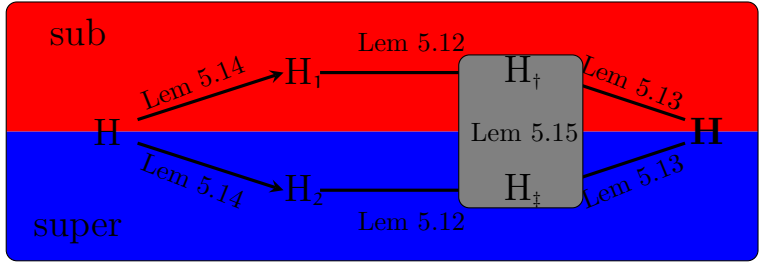
\begin{figure}[htp]
	\centering 
\begin{tikzpicture}[>=stealth,xscale=0.96,yscale=0.78]
	\node[coordinate] (nw) at (-6.8,2.2) {};
	\node[coordinate] (sw) at (-6.8,-2.2) {};
	\node[coordinate] (se) at (3.5,-2.2) {};
	\node[coordinate] (ne) at (3.5,2.2) {};
	\node[coordinate] (e) at (3.5,0) {};
	\node[coordinate] (w) at (-6.8,0) {};
	\fill[blue, opacity=0.2] (w) --  (e) {[rounded corners] |- (sw) -- (w)};
	\fill[red, opacity=0.2] (w) --  (e) {[rounded corners] |- (nw) -- (w)};
	\draw[->, very thick]   (-5,0.2)--node[left=1.5mm,above,sloped]{$\text{Lem~\ref{lem_h1}}$}(-3,1);
	\draw[->, very thick]   (-5,-0.2)--node[left=1.5mm,below,sloped]{$\text{Lem~\ref{lem_h1}}$}(-3,-1);
	\draw[->, very thick]   (-2.5,1)--node[above=1.5mm]{$\text{Lem~\ref{lem_1d}}$}(0,1);
	\draw[->, very thick]   (-2.5,-1)--node[below=1.5mm]{$\text{Lem~\ref{lem_1d}}$}(0,-1);
	\draw[<-, very thick]   (0.55,-1)--node[right=1.5mm,below,sloped]{$\text{Lem~\ref{lem_bfd}}$}(2.5,-0.2);
	\draw[<-, very thick]   (0.55,1)--node[right=1.5mm,above,sloped]{$\text{Lem~\ref{lem_bfd}}$}(2.5,0.2);
	\fill[gray, draw=black,rounded corners,fill opacity=0.3] (-0.6,-1.3) rectangle (1.1,1.3);
	\node at (-5.5,0){$\Large \text{ H}$};
	\node at (-2.8,1){$\Large \text{ H}_1$};
	\node at (-2.8,-1){$\Large \text{ H}_2$};
	\node at (0.2,-1){$\Large \text{ H}_\ddagger$};
	\node at (0.2,1){$\Large \text{ H}_\dagger$};
	\node at (2.7,0){$\Large \textbf{ H}$};
	\node at (-5.9,1.7){$\Large \text{ sub}$};
	\node at (-5.9,-1.7){$\Large \text{ super}$};
	\node at (0.3,0){$\text{Lem~\ref{Hd}}$};
	\draw[rounded corners] (nw) rectangle (se);
	\end{tikzpicture}
	\caption{Add lemmas in Figure \ref{3}}
	\label{4}
\end{figure}
\begin{proof}[Proof of \Cref{pro_compa_prin}]  
Fix $h_1$, $h_2\in C_b(M)$ and $\lambda >0$. Let $u$ be a viscosity subsolution to $(1-\lambda H)f=h_1$ and $v$ be a viscosity supersolution to $(1-\lambda H)f=h_2$. By  \Cref{lem_h1} and \Cref{lem_1d}, the function $u$ is a viscosity subsolution to $(1-\lambda H_\dagger)f=h_1$ (see red part on \Cref{4}) and $v$ is a viscosity supersolution to $(1-\lambda H_\ddagger)f=h_2$ (see blue part on \Cref{4}). Hence by the comparison principle for $H_\dagger$, $H_\ddagger$ established in \Cref{Hd}, we get $\sup_x u(x)-v(x)\leq \sup_{x}h_1(x)-h_2(x)$. This finished the proof.
\end{proof}
\subsection{Proof of \Cref{thm_viscosity_solution}: variational resolvent is a viscosity solution}\label{sec:variational resolvent is a viscosity solution}
In this section, we will prove 
\Cref{thm_viscosity_solution}. In other words, we show that for $h\in C_b(E)$ and $\lambda>0$, 
\begin{equation*}
\begin{split}
	\mathbf{R}(\lambda)h(x):=\sup_{\gamma\in \mathcal{AC}\atop \gamma(0)=x}\left\{\int^\infty_0 \lambda^{-1}e^{-\lambda^{-1}t}\left(h(\gamma(t))-\int^t_0\mathcal{L}(\gamma(r),\dot{\gamma} (r))\mathrm{d}r\right)\mathrm{d}s\right\}.
 \end{split}
\end{equation*} is a viscosity solution of the HJB equation $f-\lambda \mathbf{H}f=h$. For this strategy, we cite  \cite[Chapter 8]{FK2006} or a simplified version in \cite[Section 4]{KS2021}.
One need to check three properties of $\mathbf{R}(\lambda)$:
\begin{enumerate}
    \item \label{item_Ra} For $(f,g)\in \mathbf{H}$, we have $f=\mathbf{R}(\lambda)(f-\lambda g)$.
    \item \label{item_Rb} The operator $\mathbf{R}(\lambda)$ is a pseudo-resolvent: for all $h\in C_b(M)$ and $0<\alpha<\beta$ we have
    \begin{equation*}
        \mathbf{R}(\beta)h= \mathbf{R}\left(\mathbf{R}(\beta)h-\alpha \frac{\mathbf{R}(\beta)h-h}{\beta}\right).
    \end{equation*}
    \item \label{item_Rc} The operator $\mathbf{R}(\lambda)$ is contractive.
\end{enumerate}
Establishing \cref{item_Rc} is straightforward. The proof of \cref{item_Ra} and \cref{item_Rb} was carried out in Chapter 8 of \cite{FK2006} on the basis three assumptions, namely \cite[Assumptions 8.9,8.10 and 8.11]{FK2006}. These proofs are already complete in Euclidean space. This paper extends all proofs to Riemannian manifolds. Recalling this strategy, we conclude that it suffices to prove 
\Cref{pro_three_conditions} and \Cref{pro_pro_three_conditions_super}.
\par
We begin with the first proposition.
\begin{proposition}\label{pro_three_conditions}
	Suppose that $\mathcal{H}: T^*M\to \mathbb{R}$ is convex in $p$ and define $\mathcal{L}$ as its Lagrangian transform. Suppose that there is a good containment function $\Upsilon$ for $\mathcal{H}$. Then
	\begin{enumerate}
		\item \label{item_compact}the function $\mathcal{L}: TM\to [0,\infty]$ is lower semi-continuous and for each compact set $K\subseteq M$ and $c\in \mathbb{R}$ the set
		\begin{equation*}
			\{(x,v)\in TM ~|~x\in K,~ \mathcal{L}(x,v)\leq c\}
		\end{equation*}
		is compact in $TM$.
		\item \label{item_LleqC}for each compact $K\subseteq M$, any finite time $T>0$ and finite bound $C\geq 0$, there exists a compact set  $\hat{K}=\hat{K}(K,T,C)\subseteq M$ such that $x\in \mathcal{AC}(M)$ and $x(0)\in K$, if
		\begin{equation*}
			\int^T_0\mathcal{L}(x(s),\dot{x}(s))\mathrm{d}s\leq C,
		\end{equation*}
		then $x(t)\in \hat{K}$ for all $0\leq t\leq T$.
        \item \label{item:control_Psi_on_L_sets}
        
        for each compact set $K \subseteq M$ and $c \in \bR$, there exists a right-continuous non-decreasing function $\psi_{K,c}:\mathbb{R}^+\to \mathbb{R}^+$ such that 
        \begin{equation}\label{eqn_Psi}
			\lim_{r\to \infty}r^{-1}\psi_{K,c}(r)=0.
		\end{equation}
        and
        \begin{equation*} \label{eqn:definition:}
			|\mathrm{d}f(x) {v}|\leq \psi_{K,c}(\mathcal{L}(x,v)), \qquad \forall \, (x,v) \in TM,~x \in K,
		\end{equation*}
        where $f \in C_{K,c}$ and
        \begin{equation}\label{eqn:def_C_Kc}
            C_{K,c} := \left\{f\in C^{\infty}_c(M) \, \middle| \, \forall \,x \in K,\,|df(x)|\leq c  \right\}.
        \end{equation}
      
	\end{enumerate}
\end{proposition}
\begin{proof}
	To obtain \Cref{item_compact}, observe that $\mathcal{L}\geq 0$ follows from $\mathcal{H}(x,0)=0$. The Lagrangian $\mathcal{L}$ is convex, and lower semicontinuous as it is the Legendre transform of $\mathcal{H}$, hence in particular continuous. For $C\geq 0$, we prove that the set $\{(x,v)\in TM :x\in K, \mathcal{L}(x,v)\leq C\}$ is bounded, and hence is relatively
	compact. For any $p \in T^*_xM$ and $v\in T_xM$, we have 
	\begin{equation*}
		p v \leq \mathcal{L}(x,v) + \mathcal{H}(x,p)~~x\in K. 
	\end{equation*} Thereby, if
	$\mathcal{L}(x,v) \leq C$, then \begin{equation*}
		|v| = \sup \limits_{|p|=1}
		p  v \leq  \sup\limits_{|p|=1}
		[\mathcal{L}(x,v) + \mathcal{H}(x,p)] \leq C + C_1,    
	\end{equation*} where
	$C_1$ exists due to continuity of $\mathcal{H}$ and $x\in K$. Then for $R := C + C_1$,
	\begin{equation*}
		\{(x,v)\in TM: \mathcal{L}(x,v)\leq C\} \subseteq \{v :
		|v| \leq R\},
	\end{equation*}
	thus $\{\mathcal{L} \leq C\}$ is a bounded subset in $TM$.
	\par
	For \cref{item_LleqC}, recalling that the level sets of $\Upsilon$ are compact, we control the growth of $\Upsilon$. For $K\subseteq M$, $T>0$, $C\geq 0$ and $x\in \mathcal{AC}(M)$ as above, this follows by noting that 
	\begin{align*} \Upsilon(x(t))&=\Upsilon(x(0))+\int^t_0\mathrm{d} \Upsilon(x(s))\dot{x}(s)\mathrm{d}s\\
		&\leq\Upsilon(x(0))+\int^t_0\left[\mathcal{L}(x(s),\dot{x}(s))+\mathcal{H}(x(s),\mathrm{d} \Upsilon(x(s)))\right]\mathrm{d}s\\
		&\leq \sup_{y\in K}\Upsilon(y)+C_1+T\sup_{z\in M}\mathcal{H}(z,\mathrm{d}\Upsilon(z))=C<\infty,
	\end{align*}
	for any $0\leq t \leq T$, so that the compact set $\hat{K}=\{z\in M:\Upsilon(x)\leq C\}$ satisfies the condition.
	\par
 Proof of \ref{item:control_Psi_on_L_sets} is inspired by that of Lemma 10.21 of \cite{FK2006}. We first prove that $\mathcal{L}(x,v)$ is superlinear. Recall that $\mathcal{H}$ is continuous, which implies
    \begin{equation*}
        \overline{H}_K(c) := \sup_{x \in K} \, \, \sup_{p \in T_x^*M, |p| \leq c} \cH(x,p) < \infty.
    \end{equation*}
    Using the definition of $\cL$, it thus follows for any $(x,v) \in TM$, $x \in K$ with $|v| > 0$ that
    \begin{equation*}
        \frac{\cL(x,v)}{|v|} \geq \sup_{p \in T_x^*M, \, |p| \leq c} \frac{pv}{|v|} - \frac{\overline{H}_K(c)}{|v|} = c - \frac{\overline{H}_K(c)}{|v|}
    \end{equation*}
    It follows that
    \begin{equation*}
        \lim_{N \uparrow \infty} \, \,   \inf_{x \in K} \, \,  \inf_{v \in T_xM: |v| = N} \frac{\cL(x,v)}{|v|} = \infty.
    \end{equation*}

	Secondly, for $s\geq 0$, define the map $\vartheta(s)$ by
	\begin{equation}\label{eqn_Lv}
		\vartheta(s):= s \inf_{x \in K} \inf_{v \in T_xM: |v| \geq s}\frac{\cL(x,v)}{|v|}.
	\end{equation}
	It thus follows that $\vartheta$ is a strictly increasing function satisfying
    \begin{equation}\label{eqn:divergence_vartheta}
        \lim_{s \uparrow \infty} \frac{\vartheta(s)}{s} = \infty.
    \end{equation}
 Next, define $\Psi_{K,c}(r):C_{K,c}\vartheta^{-1}(r)$ with $\vartheta^{-1}(r)=\inf\{\omega:\vartheta(\omega)\geq r\}$. By monotonicity of $\vartheta$, we have for any $x \in K$ that
	\begin{equation*}
		\vartheta(C^{-1}_{K,c}|\mathrm{d}f(x) v|)\overset{\eqref{eqn:def_C_Kc}}\leq \vartheta(|v|) \overset{\eqref{eqn_Lv}}\leq \cL(x,v).
	\end{equation*}
	Hence by monotonicity of $\Psi_{K,c}$, we find $|\mathrm{d} f(x) v|\leq \Psi_{K,c}(\cL(x,v))$ for any $f \in C_{K,c}$, and $(x,v) \in TM$ with $x \in K$. Finally \eqref{eqn_Psi} follows by \eqref{eqn:divergence_vartheta} and the definition of $\vartheta^{-1}(r)$.
\end{proof}
\begin{remark}
\Cref{pro_three_conditions} can be used to obtain the path-space compactness of the set of trajectories that start in a compact set and have uniformly bounded Lagrangian cost. The compactness of this set can be used to get the two properties of $\mathbf{V}$ and $\mathbf{R}$.
\begin{itemize}
	\item the resolvents approximate the semigroups as in the Crandall–Ligget theorem, see \cite[Lemma 8.18]{FK2006},
	\begin{equation*}
		\lim_{n\to \infty} \mathbf{R}(t/n)^{n}f(x) = \mathbf{V}(t) f(x); 
	\end{equation*}
	\item the resolvent $\mathbf{R}(\lambda)h$ is a viscosity subsolution to the Hamilton-Jacobi-Bellman equations $f-\lambda \mathbf{H}f=h$, $\lambda> 0$ and $h \in C_b(M)$.
\end{itemize}
\end{remark}
As a by-product of \Cref{pro_three_conditions}, we have a corollary that will be used later.
\begin{corollary} \label{lemma:control_on_local_curves_compactness} Under \Cref{ass_bfH_bounded},
let $K_0\subseteq M$ a compact set and $T > 0$. For any $f\in \mathcal{D}(\mathbf{H})$, then there is a compact set $\hat{K} \subseteq M$ such that any curve $x : [0,T_0)\to M$ with $T_0 \leq T$ satisfying $x(0) \in K_0$ and for all $t < T_0$
\begin{equation} \label{eqn:satisfy_Young}
    \int_0^t \mathbf{H}f(x(s)) \dd s + \int_0^t \cL(x(s),\dot{x}(s)) \dd s = \int_0^t \dd f(x(s)) \dot{x}(s) \dd s
\end{equation}
it holds that $x(t) \in \hat{K}$ for any $t < T_0$.
\end{corollary}

\begin{proof}
    Write $c_f := \sup_{z} \left\{ -\mathbf{H}f(z) \right\}$ and $\|f\|=\sup_x|f(x)|$. Note that by \eqref{eqn:satisfy_Young}, we have for any curve
\begin{equation*}
    \int_0^t \cL(x(s),\dot{x}(s)) \dd s = f(x(t)) - f(x(0)) - \int_0^t \mathbf{H}f(x(s)) \dd s  \leq 2 \vn{f} + tc_f \leq 2 \vn{f} + T c_f
\end{equation*}
the result thus follows by Proposition \ref{pro_three_conditions} \ref{item_LleqC}.
\end{proof}
If one additionally assumes that there exists a trajectory with work $f=0$, which will follow from \Cref{pro_pro_three_conditions_super} below, one can infer that the lower semi-continuous regularization of $\mathbf{R}(\lambda)h$ is a viscosity supersolution to the Hamilton-Jacobi-Bellman equation $f-\lambda \mathbf{H}f=h$, $\lambda> 0$ and $h \in C_b(M)$.
\begin{proposition}[Global solutions' existence]\label{pro_pro_three_conditions_super}
Under \Cref{ass_bfH_bounded},
for initial point $x(0)\in M$ and any $f\in \mathcal{D}(\mathbf{H})$, there exists an absolutely continuous curve $x:[0,\infty)\to M$ and for all $0<t\leq T$ such that
	\begin{equation}\label{eqn_global_local}
		\int^t_0\mathcal{H}(x(s),\mathrm{d} f(x(s)))\mathrm{d}s+\int^t_0\mathcal{L}(x(s),\dot{x}(s))\mathrm{d}s=\int^t_0\mathrm{d}f(x(s))\dot{x}(s)\mathrm{d}s.
	\end{equation}
\end{proposition}

\subsubsection{The proof of \Cref{pro_pro_three_conditions_super}} 
\Cref{pro_pro_three_conditions_super} asserts the existence of a curve on $M$ such that \eqref{eqn_global_local} holds. This result was originally established in Euclidean space using Deimling and convex analysis, see \cite{H1993,BD1997}. However, this conclusion has not yet been independently extended to Riemannian manifolds. In this chapter, we address this gap by introducing a novel approach based on the transformation into charts to tackle this issue.
\par
 \Cref{lem_lemma1}, tells us how to transfer to chart. 
 \begin{lemma}\label{lem_lemma1}
	Let $\mathcal{O}\subset M$ be a Riemannian manifold. For an invertible smooth map $\varphi: \mathcal{O}\to \varphi(\mathcal{O}):=N$, via push-forward and pullback  in \Cref{se_Riemannian_manifold} define 
	\begin{equation*} \mathcal{H}_{\varphi}:=\mathcal{H}\circ \varphi^*:T^*N \to \mathbb{R}
	\end{equation*}
	and 
	\begin{equation*}
		\mathcal{L}_\varphi:=\mathcal{L}\circ \varphi^{-1}_*:TN\to \mathbb{R},
	\end{equation*} 
	where $\mathcal{H}:T^*M\to \mathbb{R}$ and $\mathcal{L}:TM\to \mathbb{R}$. Suppose $x : [0,\infty) \rightarrow \mathcal{O}$ is any curve such that $y(s)=\varphi(x(s))$, 
	then we have that
	\begin{enumerate}[(a)]
		\item \label{item_HL2}  
		$\mathrm{d}(f\circ\varphi^{-1})(y(s))\dot{y}(s)=\mathrm{d}f(x(s))\dot{x}(s)$,
		\item \label{item_HL3} $
		\mathcal{H}_{\varphi}(y(s),\mathrm{d} (f \circ \varphi^{-1})(y(s)))=\mathcal{H}(x(s),\mathrm{d}f(x(s)))
		$,
		\item \label{item_HL4} 
		$
		\mathcal{L}_{\varphi}(y(s),\dot{y}(s))=\mathcal{L}(x(s),\dot{x}(s))$.
		\item \label{item_HL1} $\mathcal{L}_\varphi $ is the Legendre transform of $\mathcal{H}_\varphi$, i.e.,
		$\mathcal{H}_\varphi(\eta,\xi)=\sup_{w\in T_{\eta}N}\left\{\xi(w)-\mathcal{L}_\varphi(\eta,w)\right\}$, for any $(\eta,\xi)\in T^*N$.
	\end{enumerate}
\end{lemma}

\begin{proof}
	We start to prove
	\cref{item_HL2}. By \Cref{lem_acurve}, 
	there exists a curve $x(s)$ on $\mathcal{O}$ such that
	\begin{align*}
		\mathrm{d}(f\circ\varphi^{-1})(y(s))\dot{y}(s)&=\mathrm{d}(f\circ \varphi^{-1})(\varphi(x(s)))\varphi_{*}(\dot{x}(s))\\
		&=\mathrm{d}f(x(s))\dot{x}(s),
	\end{align*}
	where in the last show we used the fact that
	\begin{equation}\label{eqn_there_a_curve}
		\begin{split}
			\mathrm{d}(f\circ \varphi^{-1})(\varphi(x(s)))&=\mathrm{d}f(\varphi^{-1}(\varphi(x(s))))\mathrm{d}(\varphi^{-1}(\varphi(x(s))))\phi_{*}(\dot{x}(s))\\
			&=\mathrm{d}f(x(s))\frac{\mathrm{d}}{\mathrm{d}t}\bigg|_{t=s}\varphi^{-1}
			(\varphi(x(t)))\\
			&=\mathrm{d}f(x(s))\dot{x}(s).
		\end{split}
	\end{equation}
	\par
	We prove \cref{item_HL3} based on the ideas when we obtain \cref{item_HL2}.
	Since
	\begin{align*}
		\mathcal{H}_{\varphi}(y(s),\mathrm{d} (f\circ \varphi^{-1})(y(s)))&=\mathcal{H}\circ \varphi^*(\varphi(x(s)),\mathrm{d}(f\circ \varphi^{-1})(\varphi(x(s)))\\
		&=\mathcal{H}(x(s),\varphi^*(\mathrm{d}(f\circ \varphi^{-1})(\varphi(x(s))),
	\end{align*}
	it suffices to prove
	\begin{align}\label{eqn_QPG}
		\varphi^*(\mathrm{d}(f\circ \varphi^{-1})(\varphi(x(s))=\mathrm{d}f(x(s)).
	\end{align}
	Indeed, applying the chain rule \eqref{eqn_chain_rule}, we imply 
	\begin{align*}
		\varphi^*(\mathrm{d}(f\circ \varphi^{-1})(\varphi(x(s))
		&=\mathrm{d}(f\circ \varphi^{-1})(\varphi(x(s))(\varphi_{*}(x(s)))\\
		&=\mathrm{d}f(\varphi^{-1}(\varphi(x(s))))\mathrm{d}(\varphi^{-1}(\varphi(x(s))))(\varphi_{*}(x(s))))\\
		&
		=\mathrm{d}f(x(s)),
	\end{align*}
	where in the last equality we use \eqref{eqn_there_a_curve}.
	Therefore, \eqref{eqn_QPG} holds. We continue to prove \cref{item_HL4} by simple calculating, and get 
	\begin{align*}
		\mathcal{L}_{\varphi}(y(s),\dot{y}(s))&=\mathcal{L}\circ \varphi^{-1}_{*}(\varphi(x(s)),\varphi_{*}(\dot{x}(s)))\\
		&=\mathcal{L}(x(s),\dot{x}(s)).
	\end{align*}
	To prove \cref{item_HL1}, for any $(\eta,\xi)\in T^*N$, we have
	\begin{align*}
		\mathcal{H}_\varphi(\eta,\xi)&=\mathcal{H}(\varphi^{-1}(\eta),\varphi^*(\xi))\\
		&=\sup_{v\in T_\eta M}\left\{\varphi^*(\xi)(v)-\mathcal{L}(\varphi^{-1}(\eta),v)\right\}\\
		&\overset{\eqref{eqn_xiX}}{=}\sup_{\varphi_*(v)\in T_{\eta}N}\left\{\xi(\varphi_*(v))-\mathcal{L}_\varphi(\eta ,\varphi_*(v))\right\}\\
		&=\sup_{w\in T_{\eta}N}\left\{\xi(w)-\mathcal{L}_\varphi(\eta,w)\right\},
	\end{align*}
	where the second equality is the fact that $\mathcal{L}$ is the Legendre transform of $\mathcal{H}$. 
	The proof is completed.
\end{proof}	
To proceed, we give the definition of subdifferential on the Euclidean space $\mathbb{R}^d$.
\begin{definition}
	For a general convex functional $p \mapsto \Phi(p)$ we denote the subdifferential at $p_0\in \mathbb{R}^d$ as the set
	\begin{equation*}
		\partial_p\Phi(p_0):=\{\xi \in \mathbb{R}^d:\Phi(p)\geq \Phi(p_0)+\xi(p-p_0),\forall p\in \mathbb{R}^d\}.
	\end{equation*}
\end{definition}
Work on charts, comparing with the result of \Cref{pro_pro_three_conditions_super}, we first get local solutions' existence.

\begin{lemma}\label{lemma:local_existence}
Let $M$ be a Riemannian manifold and let $x_0 \in M$. Consider the open ball $\cO := B_{R}(x_0)$ around $x_0$ with the radius $R > 0$ strictly smaller than the injectivity radius $i_{x_0}$ at $x_0$. Denote the normal coordinates on the ball $\cO$ by the map $\varphi : \cO \subseteq M \rightarrow \varphi(\cO) \subseteq \mathbb{R}^d$.

    Fix $f \in C^1(M)$. Then the following content holds.
    \begin{enumerate}
        \item \label{item:local_existence_on_chart} Define $f_\varphi=f\circ\varphi^{-1}$. There exists a solution $y(t) : [0,T_0(x)) \rightarrow \varphi(\cO) \subseteq \bR^d$ to the differential inclusion
    \begin{equation}\label{eqn_y_differ_inclu}
    \begin{cases}
        \dot{y}(t)\in \partial_p\mathcal{H}_\varphi(y(t),\mathrm{d}f_{\varphi}(y(t)), \\
        y(0) = 0 = \varphi(x_0)
    \end{cases}
	\end{equation}
    with 
    \begin{equation}
        T_0(x) = \inf \left\{ t > 0 \, \middle| \, y(t) \notin \varphi(B_{R/2}(x_0)) \right\}.
    \end{equation}
    \item \label{item:local_existence_integral_equality} Set $x(t) = \varphi^{-1}(y(t))$. Then the curve $x : [0,T_0(x)) \rightarrow B_{R/2}(x_0) \subseteq M$ satisfies $x(0) = x_0$ and
    \begin{equation}\label{eqn_T0V}
	\int^{t}_0\mathcal{H}(x(s),\mathrm{d} f(x(s)))\mathrm{d}s+\int^{t}_0\mathcal{L}(x(s),\dot{x}(s))\mathrm{d}s=\int^{t}_0\mathrm{d}f(x(s))\dot{x}(s)\mathrm{d}s.
	\end{equation}
    for any $t < T_0(x)$.
    \end{enumerate}
\end{lemma}
\begin{proof}
	We first prove the differential inclusion \eqref{eqn_y_differ_inclu}. By taking $\mathcal{O}=B_{R}(x_0)$ in \Cref{lem_lemma1} and define $T_0(x) = \inf \left\{ t > 0 \, \middle| \, y(t) \notin \varphi(B_{R/2}(x_0)) \right\}$, we transfer a differential inclusion on $M$ to the chart $\varphi(\mathcal{O})\subseteq \mathbb{R}^d$, and obtain \eqref{eqn_y_differ_inclu}. The subdifferential $\partial_p\mathcal{H}_\varphi(y(t),\mathrm{d}f_\varphi(y(t))$ satisfies all the conditions of Lemma 5.1 of \cite{D1992MR1189795}. Hence, there exists a solution $y(t)$ such that \eqref{eqn_y_differ_inclu} holds.
	\par
	Next, we turn to prove that there exists a solution such that \eqref{eqn_T0V} holds by local construction. To do it, for the initial point $x_0\in M$, there exists a ball $B_{R}(x_0)$ with $R$ strictly smaller than   $i_{x_0}$. By \eqref{eqn_y_differ_inclu}, we further have
\begin{equation}\label{eqn_Uy}
\int^t_0\mathcal{H}_{\varphi}(y(s),\mathrm{d} f_{\varphi}(y(s)))\mathrm{d}s+\int^t_0\mathcal{L}_{\varphi}(y(s),\dot{y}(s))\mathrm{d}s=\int^t_0\mathrm{d}f_{\varphi}(y(s))\dot{y}(s)\mathrm{d}s.
	\end{equation}
 on $\varphi(\mathcal{O})$.
 \par
 We claim that there exists a curve $y(s):[0,T_0(x))\to \varphi(\mathcal{O})$ such that \eqref{eqn_Uy} holds. 
  On the one hand, we have that
		\begin{equation*}
			\mathcal{H}_\varphi(y(s),\mathrm{d}f_{\varphi}(y(s) ))\geq \mathrm{d}f_{\varphi}(y(s))\dot{y}(s)-\mathcal{L}_\varphi(y(s),\dot{y}(s)),
	\end{equation*}
	for all $y(s)\in \varphi(\mathcal{O})$, via convex duality.
Then, integrating the above inequality gives one inequality in \eqref{eqn_Uy}.
\par Regarding the other inequality, via \eqref{eqn_y_differ_inclu} we obtain for all $p\in \varphi(\mathcal{O})$, 
\begin{equation*}
	\mathcal{H}_\varphi (y(s),p)\geq \mathcal{H}_\varphi(y(s),\mathrm{d}f_{\varphi}(y(s) ))+ \dot{y}(s) \left(p- \mathrm{d}f_{\varphi}(y(s))\right),
\end{equation*}
and as a consequence
\begin{align*}
\mathcal{H}_\varphi(y(s),\mathrm{d}f_{\varphi}(y(s) ))\leq \mathrm{d}f_{\varphi}(y(s))\dot{y}(s)-\mathcal{L}_\varphi(y(s),\dot{y}(s)),
\end{align*}
and integrating gives the other inequality.
After that by \Cref{lem_lemma1}, we transfer \eqref{eqn_Uy} on $\phi(\mathcal{O})$ to \eqref{eqn_T0V} on $M$. Then we get a curve $x(t)=\varphi^{-1}(y(s))\subseteq \mathcal{O}\subseteq M$ satisfying \eqref{eqn_T0V} for any $s < T_0(x)$. 
\end{proof}

Next showing the curve as in \Cref{lemma:local_existence} has Lagrangian cost that grows linearly in time uniformly in their starting point in a compact set.

\begin{lemma} \label{lemma:local_existence_control_on_Lagrangian}

Let $M$ be a Riemannian manifold and $K\subseteq M$ a compact set. Fix $R \in [i(K)/2,i(K))$. Then there is a constant $C$ such that for any curve $x(t) : [0,T_0(x)) \rightarrow B_{R/2}(x_0)$ with $x(0) \in K$ as in Lemma \ref{lemma:local_existence}, we have
    \begin{equation*}
        \int_0^t \cL(x(s),\dot{x}(s)) \dd s \leq C t
    \end{equation*}
    for any $t < T_0(x)$.
\end{lemma}

\begin{proof}
First of all, denote by $\hat{K}$ the compact set obtained by covering $K$ by balls of radius $R/2$. No considered curve can leave $\hat{K}$ by construction.

Denote $c_{f,\hat{K}}=\sup_{z\in\hat{K}}\{-\mathbf{H}f(z)\}$. 
As $x$ satisfies \eqref{eqn_T0V}, we have by Proposition \ref{pro_three_conditions}  \ref{item:control_Psi_on_L_sets}
\begin{align*}
    \int_0^t \cL(x(s),\dot{x}(s)) \dd s & = \int_0^t \dd f(x(s)) \dot{x}(s) \dd s - \int_0^t \mathbf{H}f(x(s)) \dd s \\
    & \leq \int_0^t \psi_{\hat{K},R} \left(\cL(x(s),\dot{x}(s)) \right)\dd s+ t c_{f,\hat{K}}.
\end{align*}
Furthermore, as $\psi_{\hat{K},R}$ is non-decreasing and the fact that $\frac{\psi_{\hat{K},R}(r)}{r}$ converges to $0$ for $r\to \infty$, there exist $0<m<1$ and $r^*\geq 1$ such that $\frac{\Psi_{\hat{K},R}(r)}{r}\leq m$ for $r \geq r^*$. Proceeding our estimate, by splitting the integral into regions $[0,t] = I_1 \cup I_2$ with
\begin{align*}
    I_1 & := \left\{s \in [0,t] \, \middle| \, \cL(x(s),\dot{x}(s)) \geq r^*\right\}, \\
    I_2 & := \left\{s \in [0,t] \, \middle| \, \cL(x(s),\dot{x}(s)) < r^*\right\},
\end{align*}
we get
\begin{align*}
    \int_0^t \cL(x(s),\dot{x}(s)) \dd s  & \leq \int_{I_1} \frac{\psi_{\hat{K},R}(\cL(x(s),\dot{x}(s)))}{\cL(x(s),\dot{x}(s))}\cL(x(s),\dot{x}(s)) \dd s +\int_{I_2} \psi_{\hat{K},R}(\cL(x(s),\dot{x}(s))) \dd s + tc_{f,\hat{K}}\\
    & \leq m \int_0^t \cL(x(s),\dot{x}(s)) \dd s + t \left( \psi_{\hat{K},R}(r^*) + c_{f,\hat{K}}\right).
\end{align*}
Rearranging terms leads to
\begin{equation*}
    \int_0^t \cL(x(s),\dot{x}(s)) \dd s \leq t \frac{\psi_{\hat{K},R}(r^*) + c_{f,\hat{K}}}{1-m}
\end{equation*}
establishing the claim with $C=\frac{\psi_{\hat{K},R}(r^*) + c_{f,\hat{K}}}{1-m}$. 
\end{proof}

Next we control the speed at which curves as in \Cref{lemma:local_existence}
move away from their starting point.

\begin{lemma} \label{lemma:control_on_local_curves}
Let $M$ be a Riemannian manifold and $K\subseteq M$ a compact set. Fix $R \in [i(K)/2,i(K))$. 

Then there is a $C > 0$ such that for any $x_0 \in K$ and any curve $x(t) : [0,T_0(x)) \rightarrow B_{R/2}(x_0)$ with $x(0) = x_0$ as in Lemma \ref{lemma:local_existence}, we have
\begin{equation*}
    \frac{1}{2}d^2(x(t),x_0) \leq t C
\end{equation*}
for any $t < T_0(x)$. In particular $T_0(x) \geq \frac{R^2}{8C}$.
\end{lemma}

We start with a preliminary lemma, that will be used in the proof of \Cref{lemma:control_on_local_curves}.

\begin{lemma} \label{lemma:smooth_distance_cutoff}
Let $K \subseteq M$ be a compact set in $M$. For any $x_0 \in K$ and radius $R < i_{x_0}$, set
\begin{equation*}
    g_{x_0,R}(x) = \theta_R\left(\frac{1}{2}d^2(x,x_0)\right)
\end{equation*}
where $\theta_R : [0,\infty) \rightarrow [0,\frac{3}{4}R]$ is a smooth non-decreasing function, satisfying $\theta'_R(r) \leq 1$ where $\theta_R(r) = r$ for $r \leq R/2$ and $\theta_R(r)$ is constant for $r \geq \frac{3}{4}R$.

For any such $R$, we have $g_{x_0,R} \in C_{K,R}$ where $C_{K,R}$ was defined in Proposition \ref{pro_three_conditions} \ref{item:control_Psi_on_L_sets}. Moreover, $g_{x_0,R}\in \mathcal{D}(\mathbf{H})$.
\end{lemma}

\begin{proof}
    By construction, we have
    \begin{equation*}
        \dd g_{x_0,R}(x) = \theta_R'\left(\frac{1}{2}d^2(x,x_0)\right) d(x,x_0)
    \end{equation*}
    which by the properties of $\theta_R$ satisfies
    \begin{equation*}
        | \dd g_{x_0,R}(x) | \leq |d(x,x_0)| \leq R
    \end{equation*}
    for any $x \in M$. In particular, we have $g_{x_0,R} \in C_{K,M}$.
    \par
    We proceed to prove $g_{x_0,R}\in \mathcal{D}(\mathbf{H})$. To do it, since  $g_{x_0,R}$ is continuously differentiable and bounded, we are left to prove that there exists a $C>0$ such that $\sup_{z}\left\{-\mathbf{H}g_{x_0,R}(z)\right\}\leq C$.
    On the one hand, we have
    \begin{equation}\label{eqn:g1}
    \begin{split}
        \int^t_0\dd g_{x_0,R}(x(s))\dot{x}(s)\dd s=g_{x_0,R}(x(t))-g_{x_0,R}(x(0))=\frac{1}{2}d^2(x(t),x_0)\geq 0.
    \end{split}
    \end{equation}
    On the other hand, we have
\begin{equation}\label{eqn:g2}
        \int^t_0\dd g_{x_0,R}(x(x))\dot{x}(s)\dd s=\int^t_0\mathcal{L}(x(s),\dot{x}(s))\dd s+\int^t_0\mathbf{H}g_{x_0,R}(x(s))\dd s. 
    \end{equation}
    Form \eqref{eqn:g1} and \eqref{eqn:g2}, we obtain that
\begin{equation}\label{eqn:g3}
    \begin{split}
        -\int^t_0\mathbf{H}g_{x_0,R}(x(s))\dd s &\leq\int^t_0\mathcal{L}(x(s),\dot{x}(s))\dd s\leq Ct,
    \end{split}
    \end{equation}
 where in the last inequality we use  \Cref{lemma:local_existence_control_on_Lagrangian} for the above estimate. Furthermore, form \eqref{eqn:g3}  we obtain $\sup_{z}\left\{-\mathbf{H}g_{x_0,R}(z)\right\}\leq C$ which deduce that $g_{x_0,R}\in \mathcal{D}(\mathbf{H})$.
\end{proof}

\begin{proof}[Proof of Lemma \ref{lemma:control_on_local_curves}]
Fix $x_0 \in K$ and any curve $x(t) : [0,T_0(x)) \rightarrow B_{R/2}(x_0)$ with $x(0) = x_0$ as in Lemma \ref{lemma:local_existence}. Let $g_{x_0,R} \in \cD(\mathbf{H})$ be any smooth bounded function as in Lemma \ref{lemma:smooth_distance_cutoff} and $\hat{K}$ be the compact set obtained by covering $K$ by balls of radius $R/2$. 
For the sake of symbol simplicity, we will write $g_{x_0,R}$ as $g$.
\par
It thus follows by Proposition \ref{pro_three_conditions} \ref{item:control_Psi_on_L_sets} and \eqref{eqn:g1} that for any $t < T_0(x)$
\begin{align*}
    \frac{1}{2}d^2(x(t),x_0) 
    & = \int_0^t \dd g(x(s)) \dot{x}(s) \dd s \\
    & \leq \int_0^t \psi_{\hat{K},R}\left( \cL(x(s),\dot{x}(s))\right) \dd s.
\end{align*}
As $\psi_{\hat{K},R}$ is non decreasing and the fact that $\frac{\psi_{\hat{K},R}(r)}{r}$ converges to $0$ for $r\to \infty$, there exist $0<m<1$ and $r^*\geq 1$ such that $\frac{\Psi_{\hat{K},R}(r)}{r}\leq m$ for $r \geq r^*$. Proceeding our estimate, by splitting the integral into regions $[0,t] = I_1 \cup I_2$ with
\begin{align*}
    I_1 & := \left\{s \in [0,t] \, \middle| \, \cL(x(s),\dot{x}(s)) \geq r^*\right\}, \\
    I_2 & := \left\{s \in [0,t] \, \middle| \, \cL(x(s),\dot{x}(s)) < r^*\right\},
\end{align*}
we get
\begin{align*}
    \frac{1}{2}d^2(x(t),x_0) & \leq \int_{I_1} \frac{\psi_{\hat{K},R}(\cL(x(s),\dot{x}(s)))}{\cL(x(s),\dot{x}(s))}\cL(x(s),\dot{x}(s)) \dd s +\int_{I_2} \psi_{\hat{K},R}(\cL(x(s),\dot{x}(s))) \dd s \\
    & \leq m \int_0^t \cL(x(s),\dot{x}(s)) \dd s + t \psi_{\hat{K},R}(r^*).
\end{align*}
By \Cref{lemma:local_existence_control_on_Lagrangian}, we conclude that
\begin{equation*}
    \frac{1}{2}d^2(x(t),x_0)  \leq t\left(mC_1+ \psi_{\hat{K},R}(r^*)\right).
\end{equation*}
The result thus follows for $C = mC_1 + \psi_{\hat{K},R}(r^*)$.   
\end{proof}

We are ready to prove \Cref{pro_pro_three_conditions_super}.
\begin{proof}[Proof of \Cref{pro_pro_three_conditions_super}]
    We argue by contradiction. Fix $x_0 \in M$ and $T > 0$. Suppose there does not exist an absolutely continuous curve $x(t)$, $t\in [0,T]$ started at $x_0 \in M$ such that \eqref{eqn_global_local} holds.

    In other words,
    \begin{equation} \label{eqn:def_maximal_time_existence}
        T_{\max{}} = \sup\left\{ T_0(x) \, \middle| \, \exists \, x : [0,T_0(x)) \rightarrow M \text{ satisfying }  \eqref{eqn_global_local}, x(0) = x_0 \right\} \leq T.
    \end{equation}
    By \Cref{lemma:control_on_local_curves_compactness} there is a compact set $K \subseteq M$ such that any curve considered in \eqref{eqn:def_maximal_time_existence} stays in $K$. Fix $\varepsilon<\frac{R^2}{8C}\leq T_0(x)$ as in Lemma \ref{lemma:control_on_local_curves}.

    Patching the curve $\tilde{x} : [0,T_0(\tilde{x}))$ started from $x(T_0(x) - \varepsilon)$ obtained from Lemma \ref{lemma:local_existence} to the curve $x$ at time $T_0 - \varepsilon$, we obtain from \Cref{lemma:control_on_local_curves} that this curve, is a solution to \eqref{eqn_global_local} on the time interval $[0,T_0(x) -\varepsilon + T_0(\tilde{x})$, which contradicts \eqref{eqn:def_maximal_time_existence}.

    This establishes the claim.  
\end{proof}

Here, we turn to prove \Cref{thm_viscosity_solution}.

\begin{proof}[Proof of \Cref{thm_viscosity_solution}]
Recalling the analysis at the beginning of \Cref{sec:variational resolvent is a viscosity solution}. Items \ref{item_Ra} and \ref{item_Rb} were obtained by Propositions \ref{pro_three_conditions} and \ref{pro_pro_three_conditions_super}. This, together with \cref{item_Rc}, the resolvent $\mathbf{R}(h)$ is a viscosity solution for $f-\lambda \mathbf{H}f=h$ for $\lambda>0$ and $h\in C_b(M)$. 
\end{proof}

\subsection{Variational representation of the Hamiltonian on Riemannian manifold }\label{se7}
 In this section, we will prove our main result \Cref{thm_LDPmanifold}. According to the strategy of proof of \Cref{thm_LDPmanifold} in \Cref{sec:The_strategy} and the \Cref{Fig:The_strategy_of proof_Thm}, the result of LDP with a rate function \eqref{eqn_rate_function_1} is based on three main parts that we have proven:
 \begin{itemize}
     \item Operator convergence;
    \item Exponential tightness;
    \item Comparison principle.
 \end{itemize}
We are left to give an action-integral rate function form \eqref{eqn_rate_function_1} for completing the proof of \Cref{thm_LDPmanifold}.
 \begin{proof}[Proof of \Cref{thm_LDPmanifold}]
 The large deviation principle with an action-integral rate function can be obtained from
\Cref{pro_convergence_of_operator,lem_eigen,pro_expo,pro_compa_prin} and \Cref{thm_viscosity_solution}. 
 Based on the proof of \cite[Theorem 8.14]{FK2006} that if semigroups $V(t)$ and $\mathbf{V}(t)$ agree, then the rate function 
  \begin{equation*}
	I(\gamma)=
	\begin{cases}
		I_0(\gamma(0))+\int^\infty_0\mathcal{L}\left(\gamma(s),\dot{\gamma}(s)\right)\mathrm{d}s,&\mbox{if~$\gamma\in\mathcal{AC}(M)$,}\\
		\infty,&\mbox{otherwise.}
	\end{cases}
\end{equation*}
  Then the proof is completed.
 \end{proof}

\appendix
\section{Riemannian manifold}\label{se_Riemannian_manifold}
In this section, we introduce some fundamental deﬁnitions, properties, and notation. This can be found in any textbook on Riemannian manifold, for example, \cite{L2003,W2014MR3154951}.
\par
Throughout the paper, $(M,g)$ is a $d$-dimensional connected complete Riemannian manifold. We start with the definition of \textit{chart}, which is used for the existence proof below, \Cref{pro_pro_three_conditions_super}.
For each open set $\mathcal{O} \subset M$, if $\varphi:\mathcal{O}\to \mathbb{R}^d$ is a homeomorphism onto an open subset of $\mathbb{R}^d$, then $M$ is called a $d$-dimensional topological manifold and $(\mathcal{O},\phi)$ is called a coordinate neighborhood on $M$. A $d$-dimensional differential structure on $M$ is a family $\mathcal{U}:= \{ (\mathcal{O}_\alpha, \varphi_\alpha) \}$ of coordinate neighborhoods such that
\begin{enumerate}[(i)]
	\item $\cup_\alpha\mathcal{O}_\alpha \supset M$;
	\item For any $\alpha$, $\beta$, $\varphi_\alpha \circ \varphi^{-1}_\beta:\varphi_\beta (\mathcal{O}_\beta \cap \mathcal{O}_\alpha)\to \varphi_\alpha (\mathcal{O}_\beta \cap \mathcal{O}_\alpha )$ is $C^\infty$-smooth, i.e. $(\mathcal{O}_\alpha,\varphi_\alpha)$ and $(\mathcal{O}_\beta,\varphi_\beta)$ are $C^\infty$-compatible;
	\item If a coordinate neighborhood $(\mathcal{O},\varphi)$ is $C^\infty$-compatible with each $(\mathcal{O}_\alpha,\varphi_\alpha)$ in $\mathcal{U}$, then $(\mathcal{O},\varphi)\in \mathcal{U}$.
\end{enumerate}
Each $(\mathcal{O},\varphi)\in \mathcal{U}$ is called a local (coordinate) chart. 
\par
The tangent space of $M$ at $x\in M$ is denoted by $T_xM$. We denote by $\langle \cdot,\cdot\rangle_x=g(\cdot,\cdot)$ the scalar product on $T_xM$ with the associated norm $|\cdot|_x$, where the subscript $x$ is sometimes omitted. The tangent bundle of $M$ is denoted by $TM:=\cup_{x\in M}T_xM$, which is naturally a manifold. Let $T^*_xM=(T_xM)^*$ be the \textit{cotangent space} at $x\in M$, namely the dual space of the tangent space $T_xM$ (the space of linear functions on $T_xM$).
Let $T^*M=\cup_{x\in M}T^*_xM$, which is called the \textit{cotangent bundle} on $M$.
\par 
Given a piecewise smooth curve $\gamma:[a, b] \to M$ joining $x$ to $y $ (i.e. $\gamma (a) = x$ and $\gamma(b)=y$, we can define the length of $\gamma$ by $l(\gamma) = \int^b_a |\dot{\gamma}(t)|\mathrm{d}t$. Then the Riemannian distance $d(x,y)$, which induces the original topology on $M$, is defined by minimizing this length over the set of all such curves joining $x$ to $y$.

Let $\nabla$ be the Levi-Civita connection associated with the Riemannian metric. Let $\gamma$ be a smooth curve in $M$. A vector ﬁeld $X$ is said to be parallel along $\gamma$ if and only if $\nabla_{\dot{\gamma}_t} X = 0$. If $\dot{\gamma}$ itself is parallel along $\gamma$, we say that $\gamma$ is a geodesic, and in this case $|\dot{\gamma}|$ is constant. When $ |\dot{\gamma}|= 1$, $\gamma$ is said to be normalized. A geodesic joining $x$ to $y$ in $M$ is said to be minimal if its length equals $d(x,y)$.
\par
A Riemannian manifold is complete if for any $x\in M $ all geodesics emanating from $x$ are deﬁned for all $-\infty< t <\infty$. By the Hopf-Rinow Theorem \cite[Theorem 6.13]{L1997}, we know that if M is complete then any pair of points in M can be joined by a minimal geodesic. Moreover, $(M,d)$ is a complete metric space and bounded closed subsets are compact.

Given a (piecewise) smooth curve $\gamma : [a, b] \to M$, we denote parallel transport along $\gamma$ from $\gamma (t_0 )$ to $\gamma (t_1 )$ by $\tau_{\gamma,t_0t_1}$, or simply $\tau_{t_0t_1}$ whenever the meant curve is clear. If points $x$, $y\in M$ can be connected by a unique geodesic of minimal length, we will also write $\tau_{xy}$ meaning parallel transport from $x$ to y along this specific geodesic.

The exponential map exp $x: T_x M\to M$ at $x$ is deﬁned by $\exp_x v = \gamma_v(1,x)$ for each $v\in T_xM$, where $\gamma (\cdot) =\gamma_v (\cdot,x)$ is the geodesic starting at $x$ with velocity $v$. Then $\exp_x(tv) =\gamma_v(t,x)$ for each real number $t$. Note that the mapping $\exp_x$ is differentiable on $T_xM$ for any $x\in M$.

In many cases, the minimal geodesic is not unique. For instance, for the unit sphere $\mathbb{S}^d $, each half circle linking the highest and the lowest points is a minimal geodesic. This fact leads to the notion of cut-locus.

\begin{definition}
	Let $x\in M$. For any $X\in \mathbb{S}_x:=\{X\in T_xM:~|X|=1\}$, let
	\begin{equation*}
		r(X):=\sup\{t>0:d(x,~\exp_x(tX))=t\}.
	\end{equation*}
	If $r(X)<\infty$ then we call $\exp_x(r(X)X)$ a cut-point of $x$. The set
	\begin{equation*}
		\mathrm{cut}(x):=\{\exp_x(r(X)X)~:~X\in \mathbb{S}_x,~r(X)<\infty\}
	\end{equation*}
	is called the \textit{cut-locus} of the point of $x$. Moreover, the quantity
	\begin{equation*}
		i_x:=\inf\{r(X):X\in \mathbb{S}_x\}
	\end{equation*}
	is called the \textit{injectivity radius} of $x$. For any set $A \subseteq M$ we write $i(A) := \inf_{x\in A} i_x $ the injectivity radius of $A$.
\end{definition}

\begin{lemma}[\cite{K1982}] \label{lemma:compact_positive_injectivity_radius}
	The injectivity radius $i_x$ depends continuously on $ x$. In particular, if $K \subseteq M$ is compact we have $i(K) > 0$. 
\end{lemma}

Note that $i(K)>0$ is used to find a smooth distance on $M$.
\begin{definition}
	Let $\mathcal{T}(M)$ be the space of smooth vector ﬁelds on $M$ and let $\nabla$ be any connection on $M$. The formula
	\begin{equation*}
		\mathcal{R}(X,Y)Z:=\nabla_Y\nabla_XZ-\nabla_X\nabla_YZ+\nabla_{[X,Y]}Z,
	\end{equation*}
	for $X$, $Y$, $Z\in \mathcal{T}(M)$, defines a function
	$\mathcal{R}:\mathcal{T}(M)\times \mathcal{T}(M)\times \mathcal{T}(M)$ called the Riemannian curvature of $M$, where $[X,Y]=XY-YX$ is the commutator of $X$ and $Y$.
\end{definition}
\par
To prove the existence solutions of HJB equations on $M$, we need the definitions of push-forward and pullback.
\begin{definition}[Push-forward]
	If $M$ and $N$ are smooth manifolds and $\varphi:M\to N$ is a smooth map, for each $p\in M$ we define a map 
	\begin{equation}\label{eqn_push_forward}
		\varphi_{*p}:T_pM\to T_{\varphi(p)}N,
	\end{equation} called the \textit{push-forward} associated with $\varphi$, by
	\begin{equation*}
		(\varphi_{*p}(v))(f)=v(f\circ \varphi),~~~v\in T_pM,~f\in C^\infty(M).
	\end{equation*}
	
\end{definition}
\begin{definition}[Pullback]
	Let 
	$\varphi: M\to N$ be an invertible smooth map, and let $p\in M$ be arbitrary. 
	A dual map by pullback associated with $\varphi$,
	\begin{equation*}
		\varphi^*_p:T^*_{\varphi(p)} N \to T^*_pM.
	\end{equation*}
	Moreover, $\varphi^{*}_p$ is characterized by
	\begin{equation}\label{eqn_xiX}
		(\varphi^{*}_p\xi)(v)=\xi(\varphi_{*p}(v)), ~~\mbox{for}~\xi \in T^*_{\varphi(p)} N,~v\in T_pM.
	\end{equation}
	We can put all $\varphi_{*p} $ and $\varphi^*_p$ together to obtain $\varphi_*:TM\to TN$ and $\varphi^*:T^*N \to T^*M$, respectively.
	
\end{definition}
The next lemma shows that tangent vectors to curves behave well under composition with smooth maps. 
\begin{lemma}[Proposition 3.11 in \cite{L2003}]\label{lem_acurve}
	Let $\varphi:M\to N$ be a smooth map, and let $\gamma:J \to M$ be a smooth curve, where $J\in \mathbb{R}$ is an interval. For any $t\in J$, the tangent vector to the composite curve $\varphi \circ \gamma$ at $t = t_0$ is given by
	\begin{equation*}
		(\dot{\varphi\circ \gamma})(t_0)=(\varphi\circ \gamma)_*\frac{\mathrm{d}}{\mathrm{d}t}\bigg|_{t_0}=\varphi_*\dot{\gamma}(t_0).
	\end{equation*}
\end{lemma}
The chain rule for total derivatives is important in Riemannian manifolds because it allows us to compute the derivative of a composite function.
\begin{lemma}[The chain rule for total derivatives, Proposition A.24 in \cite{L2003}]
	Suppose $V$, $W$, $X$ are finite-dimensional vector spaces, $U\subset V$ and $\Tilde{U}\subset W$ are open sets, and $F:U \to\Tilde{U}$ and $G:\Tilde{U}\to X$ are maps. If $F$ is differentiable at $a \in U$ and $G$ is diifferentiable at $F(a)\in U$, then $G\circ F$ is differentiable at $a$, and
	\begin{equation}\label{eqn_chain_rule}
		D(G \circ F)(a) = DG(F(a)) \circ DF(a). 
	\end{equation}
\end{lemma}

\section{Viscosity solutions}
Now we define viscosity sub and supersolutions, which is often used in the proof.
\begin{definition}[Viscosity solutions] \label{def_VS}Let $H\subseteq C_b(M)\times C_b(M\times S)$ be a multivalued operator. We denote $\mathcal{D}(H)$ for the domain of $H$ and $\mathcal{R}(H)$ for the range of $H$. Let $\lambda>0$ and $h\in C_b(M)$. Consider the Hamilton-Jacobi equations
	\begin{equation}\label{HJ}
		f-\lambda H f=h.
	\end{equation}
	\begin{description}
		\item [Classical solutions] 
		We say that $u$ is a classical subsolution of \eqref{HJ} if there is a $g$ such that $(u,g)\in H$ and $u - \lambda g \leq h$. We say that $v$ is a classical supersolution of \eqref{HJ} if there is a function $g$ such that $(v, g) \in H$ and $v - \lambda g \geq h$. We say that $u$ is a classical solution if it is both a subsolution and a supersolution.
		\item[Viscosity subsolutions] We say that $u$ is a (viscosity) subsolution of \eqref{HJ} if $u$ is bounded, upper semicontinuous, and if for every $(f,g)\in H$ there exists a sequence $(x_n,z_n) \in M\times S$ such that
		\begin{equation*}
			\lim_{n\to\infty}u(x_n)-f(x_n)=\sup_{x}u(x)-f(x),
		\end{equation*}
		\begin{equation*}
			\limsup_{n\to\infty}u(x_n)-\lambda g(x_n,z_n)-h(x_n)\leq 0.
		\end{equation*}
		\item[Viscosity supersolutions] We say that $v$ is a (viscosity) supersolution of \eqref{HJ} 
		if $v$ is bounded, lower semicontinuous, and if for every $(f,g)\in H$ there exists a sequence sequence $(x_n,z_n) \in M\times S$ such that
		\begin{equation*}
			\lim_{n\to\infty}v(x_n)-f(x_n)=\inf_{x}v(x)-f(x),
		\end{equation*}
		\begin{equation*}
			\liminf_{n\to\infty}v(x_n)-\lambda g(x_n,z_n)-h(x_n)\geq 0.
		\end{equation*}
		\item [Viscosity solutions] We say that $u$ is a (viscosity) solution of \eqref{HJ} if it is both a subsolution and a supersolution to \eqref{HJ}. 
	\end{description}
\end{definition}
\begin{remark}
	Consider the definition of subsolutions. Suppose that the test function $(f,g)\in H$ has compact sublevel sets, then instead of working with a sequence $(x_n,z_n)$, we can pick $(x_0,z_0)$ such that
	\begin{equation*}
		u(x_0)-f(x_0)=\sup_{x}u(x)-f(x),
	\end{equation*}
	\begin{equation*}
		u(x_0)-\lambda g(x_0,z_0)-h(x_0)\leq 0.
	\end{equation*}
	Similarly, a simplification holds in the case of supersolutions. This is used in the proof \Cref{lem_1d} below.
\end{remark}
\begin{definition}[Comparison principle]
	We say that \eqref{HJ} satisfies the comparison principle if for every viscosity subsolutions $u$ and viscosity supersolutions $v$ to \eqref{HJ}, we have $u\leq v$.
\end{definition}

\begin{remark}[Uniqueness] 
	The comparison principle implies uniqueness of viscosity solutions. Suppose that $u$ and $v$ are both viscosity solutions, then the comparison principle yields that $u\leq v$ and $v\leq u$, implying that $u=v$.
\end{remark}

\section*{Acknowledgments}
The author thanks Rik Versendaal for numerous valuable comments.
\\[0.2cm]
The research of Y. Hu was supported by the China Scholarship Council (CSC).
The research of F. Xi was supported by the National Natural Science Foundation of China (Grant No. 12071031).

%
%
%
%
%
%
	

\printbibliography
\end{document}